\documentclass[12pt,a4paper,draft]{amsart}

\usepackage{latexsym}
\usepackage{ifthen}
\usepackage[leqno]{amsmath}
\usepackage{enumerate}
\usepackage{calc}
\usepackage{amstext,amsbsy,amsopn,amsthm,amsgen,amsfonts,amscd,amsxtra,upref}
\usepackage{mathrsfs}\usepackage{euscript}\usepackage{amssymb}

\swapnumbers
\theoremstyle{plain}
\newtheorem{Thm}{Theorem}[section]
\newtheorem{Lem}[Thm]{Lemma}
\newtheorem{Cor}[Thm]{Corollary}
\newtheorem{Pro}[Thm]{Proposition}
\newtheorem{Prp}[Thm]{Properties}
\newtheorem{Sub}[Thm]{Sublemma}

\theoremstyle{definition}
\newtheorem{Def}[Thm]{Definition}
\newtheorem{Exm}[Thm]{Example}
\newtheorem{Exs}[Thm]{Examples}

\theoremstyle{remark}
\newtheorem{Rem}[Thm]{Remark}
\newtheorem{Rms}[Thm]{Remarks}
\newtheorem*{Com}{Commentary}

\newcommand{\myEmail}{piotr.niemiec@uj.edu.pl}
\newcommand{\myAddress}{\noindent{}Piotr Niemiec\\{}Jagiellonian University\\{}Institute of Mathematics\\{}
   ul. \L{}ojasiewicza 6\\{}30-348 Krak\'{o}w\\{}Poland}
\newcommand{\myData}{\author[P. Niemiec]{Piotr Niemiec}\address{\myAddress}\email{\myEmail}}

\newcommand{\NNN}{\mathbb{N}}
\newcommand{\QQQ}{\mathbb{Q}}\newcommand{\RRR}{\mathbb{R}}

\newcommand{\CCc}{\CMcal{C}}
\newcommand{\FFf}{\CMcal{F}}\newcommand{\HHh}{\CMcal{H}}
\newcommand{\KKk}{\CMcal{K}}

\newcommand{\SSs}{\CMcal{S}}
\newcommand{\WWw}{\CMcal{W}}

\newcommand{\ZzZ}{\EuScript{Z}}
\newcommand{\Aa}{\mathfrak{A}}\newcommand{\Bb}{\mathfrak{B}}

\newcommand{\Mm}{\mathfrak{M}}\newcommand{\Nn}{\mathfrak{N}}

\newcommand{\mM}{\mathfrak{m}}

\newcommand{\SECT}[1]{\section{#1}\renewcommand{\theequation}{\thesection-\arabic{equation}}\setcounter{equation}{0}}

\newcounter{help}
\newcommand{\ITE}[3]{\ifthenelse{#1}{#2}{#3}}\newcommand{\ITEE}[3]{\ITE{\equal{#1}{#2}}{#3}{}}

\newcommand{\diam}{\operatorname{diam}}

\newcommand{\dom}{\operatorname{dom}}\newcommand{\im}{\operatorname{im}}
\newcommand{\lin}{\operatorname{lin}}
\newcommand{\card}{\operatorname{card}}

\newcommand{\Metr}{\operatorname{Metr}}

\newcommand{\scalarr}{\langle\cdot,\mathrm{-}\rangle}
\newcommand{\scalar}[2]{\langle #1,#2\rangle}\newcommand{\leqsl}{\leqslant}\newcommand{\geqsl}{\geqslant}

\newcommand{\epsi}{\varepsilon}\newcommand{\dd}{\colon}
\newcommand{\dint}[1]{\,\textup{d} #1}

\newcommand{\tfcae}{the following conditions are equivalent:}

\newcommand{\THM}[1]{Theorem~\textup{\ref{thm:#1}}}
\newcommand{\COR}[1]{Corollary~\textup{\ref{cor:#1}}}
\newcommand{\LEM}[1]{Lemma~\textup{\ref{lem:#1}}}\newcommand{\PRO}[1]{Proposition~\textup{\ref{pro:#1}}}

\newcommand{\EXM}[1]{Example~\textup{\ref{exm:#1}}}

\newenvironment{thm}[1]{\begin{Thm}\label{thm:#1}}{\end{Thm}}\newenvironment{lem}[1]{\begin{Lem}\label{lem:#1}}{\end{Lem}}
\newenvironment{cor}[1]{\begin{Cor}\label{cor:#1}}{\end{Cor}}\newenvironment{pro}[1]{\begin{Pro}\label{pro:#1}}{\end{Pro}}
\newenvironment{dfn}[1]{\begin{Def}\label{def:#1}}{\end{Def}}
\newenvironment{exm}[1]{\begin{Exm}\label{exm:#1}}{\end{Exm}}
\newenvironment{rem}[1]{\begin{Rem}\label{rem:#1}}{\end{Rem}}

\newenvironment{thmm}[2]{\begin{Thm}[#2]\label{thm:#1}}{\end{Thm}}
\newenvironment{lemm}[2]{\begin{Lem}[#2]\label{lem:#1}}{\end{Lem}}

\newcommand{\bibITEM}[2]{\ITE{\equal{#2}{}}{\bibitem{#1} }{\bibitem[#2]{#1} }}
\newcommand{\BIB}[8]{
   \bibITEM{#1}{#8} #2, \textit{#3}, #4{} \textbf{#5} (#6), #7.}
\newcommand{\myBIB}[6][P. Niemiec]{#1, \textit{#2}, #3{}\ITE{\equal{#4}{}}{}{ \textbf{#4}} (#5), #6.}
\newcommand{\BIb}[6]{
   \bibITEM{#1}{#6} #2, \textit{#3}, #4, #5.}
\newcommand{\BiB}[9]{
   \bibITEM{#1}{#9} #2, \textit{#3}, #4{} \textit{#5}, #6, #7, #8.}

\newcommand{\myBAPP}[3][P. Niemiec]{
   #1, \textit{#2}, #3}

\newcommand{\jRN}[2][]{
   \ITEE{#2}{ActaM}{\ITE{\equal{#1}{+}}
      {Acta Mathematica}{Acta Math.}}
   \ITEE{#2}{ActaMSinES}{\ITE{\equal{#1}{+}}
      {Acta Mathematica Sinica (English Series)}{Acta Math. Sin. (Engl. Ser.)}}
   \ITEE{#2}{AdvM}{\ITE{\equal{#1}{+}}
      {Advances in Mathematics}{Adv. in Math.}}
   \ITEE{#2}{ACS}{\ITE{\equal{#1}{+}}
      {Applied Categorical Structures}{Appl. Categor. Struct.}}
   \ITEE{#2}{ActaSM}{\ITE{\equal{#1}{+}}
      {Acta Scientiarum Mathematicarum}{Acta Sci. Math.}}
   \ITEE{#2}{AmJM}{\ITE{\equal{#1}{+}}
      {American Journal of Mathematics}{Amer. J. Math.}}
   \ITEE{#2}{AmMMon}{\ITE{\equal{#1}{+}}
      {American Mathematical Monthly}{Amer. Math. Mon.}}
   \ITEE{#2}{AnnSciEcNormSupT}{\ITE{\equal{#1}{+}}
      {Annales Scientifiques de l'\'{E}cole Normale Sup\'{e}rieure (3)}{Ann. Sci. \'{E}c. Norm. Sup\'{e}r. (3)}}
   \ITEE{#2}{AnnM}{\ITE{\equal{#1}{+}}
      {Annals of Mathematics}{Ann. Math.}}
   \ITEE{#2}{AnnProb}{\ITE{\equal{#1}{+}}
      {The Annals of Probability}{Ann. Probab.}}
   \ITEE{#2}{AnnPALog}{\ITE{\equal{#1}{+}}
      {Annals of Pure and Applied Logic}{Ann. Pure Appl. Logic}}
   \ITEE{#2}{APM}{\ITE{\equal{#1}{+}}
      {Annales Polonici Mathematici}{Ann. Polon. Math.}}
   \ITEE{#2}{ArchM}{\ITE{\equal{#1}{+}}
      {Archiv der Mathematik}{Arch. Math.}}
   \ITEE{#2}{AttiAccLincRendNat}{\ITE{\equal{#1}{+}}
      {Atti della Accademia Nazionale dei Lincei. Rendiconti. Classe di Scienze Fisiche, Matematiche e Naturali}
      {Atti Accad. Naz. Lincei Rend. Cl. Sci. Fis. Mat. Nat.}}
   \ITEE{#2}{BAMS}{\ITE{\equal{#1}{+}}
      {Bulletin of the American Mathematical Society}{Bull. Amer. Math. Soc.}}
   \ITEE{#2}{BAustrMS}{\ITE{\equal{#1}{+}}
      {Bulletin of the Australian Mathematical Society}{Bull. Austral. Math. Soc.}}
   \ITEE{#2}{BLondMS}{\ITE{\equal{#1}{+}}
      {Bulletin of the London Mathematical Sociecy}{Bull. Lond. Math. Soc.}}
   \ITEE{#2}{BAPolSSSM}{\ITE{\equal{#1}{+}}
      {Bulletin de l'Acad\'{e}mie Polonaise des Sciences. S\'{e}rie des Sciences Math\'{e}matiques}
      {Bull. Acad. Pol. Sci. S\'{e}r. Sci. Math.}}
   \ITEE{#2}{BullSM}{\ITE{\equal{#1}{+}}
      {Bulletin des Sciences Math\'{e}matiques}{Bull. Sci. Math.}}
   \ITEE{#2}{BullPol}{\ITE{\equal{#1}{+}}
      {Bulletin of the Polish Academy of Sciences: Mathematics}{Bull. Pol. Acad. Sci. Math.}}
   \ITEE{#2}{CanadJM}{\ITE{\equal{#1}{+}}
      {Canadian Journal Mathematics}{Canad. J. Math.}}
   \ITEE{#2}{CollectM}{\ITE{\equal{#1}{+}}
      {Collectanea Mathematica}{Collect. Math.}}
   \ITEE{#2}{CMUC}{\ITE{\equal{#1}{+}}
      {Commentationes Mathematicae Universitatis Carolinae}{Comment. Math. Univ. Carolin.}}
   \ITEE{#2}{CRParis}{\ITE{\equal{#1}{+}}
      {C. R. Paris}{C. R. Paris}}
   \ITEE{#2}{CRASParis}{\ITE{\equal{#1}{+}}
      {Comptes Rendus de l'Acad\'{e}mie des Sciences. Paris}{C. R. Acad. Sci. Paris}}
   \ITEE{#2}{CEurJM}{\ITE{\equal{#1}{+}}
      {Central European Journal of Mathematics}{Cent. Eur. J. Math.}}
   \ITEE{#2}{CMHelv}{\ITE{\equal{#1}{+}}
      {Commentarii Mathematici Helvetici}{Comment. Math. Helv.}}
   \ITEE{#2}{CollM}{\ITE{\equal{#1}{+}}
      {Colloquium Mathematicum}{Coll. Math.}}
   \ITEE{#2}{ComposM}{\ITE{\equal{#1}{+}}
      {Compositio Mathematica}{Compos. Math.}}
   \ITEE{#2}{CzMJ}{\ITE{\equal{#1}{+}}
      {Czechoslovak Mathematical Journal}{Czech. Math. J.}}
   \ITEE{#2}{DissM}{\ITE{\equal{#1}{+}}
      {Dissertationes Mathematicae}{Dissert. Math.}}
   \ITEE{#2}{DANSSSR}{\ITE{\equal{#1}{+}}
      {Doklady Akademii Nauk SSSR}{Dokl. Akad. Nauk SSSR}}
   \ITEE{#2}{DukeMJ}{\ITE{\equal{#1}{+}}
      {Duke Mathematical Journal}{Duke Math. J.}}
   \ITEE{#2}{ELA}{\ITE{\equal{#1}{+}}
      {The Electronic Journal of Linear Algebra}{Electron. J. Linear Algebra}}
   \ITEE{#2}{ExtrM}{\ITE{\equal{#1}{+}}
      {Extracta Mathematicae}{Extracta Math.}}
   \ITEE{#2}{FM}{\ITE{\equal{#1}{+}}
      {Fundamenta Mathematicae}{Fund. Math.}}
   \ITEE{#2}{FAA}{\ITE{\equal{#1}{+}}
      {Functional Analysis and its Applications}{Funct. Anal. Appl.}}
   \ITEE{#2}{FunkAnalPril}{\ITE{\equal{#1}{+}}
      {Funktsional'ny\u{\i} Analiz i Ego Prilozheniya}{Funkts. Anal. Prilozh.}}
   \ITEE{#2}{GTopA}{\ITE{\equal{#1}{+}}
      {General Topology and its Applications}{General Topol. Appl.}}
   \ITEE{#2}{HJM}{\ITE{\equal{#1}{+}}
      {Houston Journal of Mathematics}{Houston J. Math.}}
   \ITEE{#2}{IllinoisJM}{\ITE{\equal{#1}{+}}
      {Illinois Journal of Mathematics}{Illinois J. Math.}}
   \ITEE{#2}{IndagMP}{\ITE{\equal{#1}{+}}
      {Indagationes Mathematicae (Proceedings)}{Indagationes Math. Proc.}}
   \ITEE{#2}{IndianaUMJ}{\ITE{\equal{#1}{+}}
      {Indiana University Mathematical Journal}{Indiana Univ. Math. J.}}
   \ITEE{#2}{InHauEtSPM}{\ITE{\equal{#1}{+}}
      {Inst. Hautes \'{E}tudes Sci. Publ. Math.}{Inst. Hautes \'{E}tudes Sci. Publ. Math.}}
   \ITEE{#2}{IEOT}{\ITE{\equal{#1}{+}}
      {Integral Equations and Operator Theory}{Integr. Equ. Oper. Theory}}
   \ITEE{#2}{IsraelJM}{\ITE{\equal{#1}{+}}
      {Israel Journal of Mathematics}{Israel J. Math.}}
   \ITEE{#2}{JAusMSA}{\ITE{\equal{#1}{+}}
      {Journal of the Australian Mathematical Society. Series A}{J. Aust. Math. Soc. Ser. A}}
   \ITEE{#2}{JCA}{\ITE{\equal{#1}{+}}
      {Journal of Convex Analysis}{J. Convex Anal.}}
   \ITEE{#2}{JChinUST}{\ITE{\equal{#1}{+}}
      {J. China Univ. Sci. Tech.}{J. China Univ. Sci. Tech.}}
   \ITEE{#2}{JFA}{\ITE{\equal{#1}{+}}
      {Journal of Functional Analysis}{J. Funct. Anal.}}
   \ITEE{#2}{JKoreanMS}{\ITE{\equal{#1}{+}}
      {Journal of the Korean Mathematical Society}{J. Korean Math. Soc.}}
   \ITEE{#2}{JMAnApp}{\ITE{\equal{#1}{+}}
      {J. Math. Anal. Appl.}{J. Math. Anal. Appl.}}
   \ITEE{#2}{JOT}{\ITE{\equal{#1}{+}}
      {Journal of Operator Theory}{J. Operator Theory}}
   \ITEE{#2}{KodaiMSemRep}{\ITE{\equal{#1}{+}}
      {Kodai Math. Sem. Rep.}{Kodai Math. Sem. Rep.}}
   \ITEE{#2}{LAA}{\ITE{\equal{#1}{+}}
      {Linear Algebra and its Applications}{Linear Algebra Appl.}}
   \ITEE{#2}{LMLA}{\ITE{\equal{#1}{+}}
      {Linear and Multilinear Algebra}{Linear Multilinear Algebra}}
   \ITEE{#2}{LNM}{\ITE{\equal{#1}{+}}
      {Lecture Notes in Mathematics}{Lecture Notes Math.}}
   \ITEE{#2}{MathJap}{\ITE{\equal{#1}{+}}
      {Math. Japon.}{Math. Japon.}}
   \ITEE{#2}{MLQ}{\ITE{\equal{#1}{+}}
      {Mathematical Logic Quarterly}{Math. Log. Q.}}
   \ITEE{#2}{MProcCambPhS}{\ITE{\equal{#1}{+}}
      {Mathematical Proceedings of the Cambridge Philosophical Society}{Math. Proc. Cambridge Phil. Soc.}}
   \ITEE{#2}{MMag}{\ITE{\equal{#1}{+}}
      {Mathematics Magazine}{Math. Mag.}}
   \ITEE{#2}{MSb}{\ITE{\equal{#1}{+}}
      {Matematicheski\u{\i} Sbornik}{Mat. Sb.}}
   \ITEE{#2}{MStud}{\ITE{\equal{#1}{+}}
      {Matematychni Studi\"{\i}}{Mat. Stud.}}
   \ITEE{#2}{MScand}{\ITE{\equal{#1}{+}}
      {Mathematica Scandinavica}{Math. Scand.}}
   \ITEE{#2}{MAnn}{\ITE{\equal{#1}{+}}
      {Mathematische Annalen}{Math. Ann.}}
   \ITEE{#2}{MAMS}{\ITE{\equal{#1}{+}}
      {Memoirs of the American Mathematical Society}{Mem. Amer. Math. Soc.}}
   \ITEE{#2}{MichMJ}{\ITE{\equal{#1}{+}}
      {Michigan Mathematical Journal}{Mich. Math. J.}}
   \ITEE{#2}{MonatM}{\ITE{\equal{#1}{+}}
      {Monatshefte f\"{u}r Mathematik}{Mh. Math.}}
   \ITEE{#2}{NonlinA}{\ITE{\equal{#1}{+}}
      {Nonlinear Analysis: Theory, Methods \& Applications}{Nonlinear Anal.}}
   \ITEE{#2}{NAMS}{\ITE{\equal{#1}{+}}
      {Notices of the American Mathematical Society}{Notices Amer. Math. Soc.}}
   \ITEE{#2}{OpusM}{\ITE{\equal{#1}{+}}
      {Opuscula Mathematica}{Opuscula Math.}}
   \ITEE{#2}{PacJM}{\ITE{\equal{#1}{+}}
      {Pacific Journal of Mathematics}{Pacific J. Math.}}
   \ITEE{#2}{PeriodMHung}{\ITE{\equal{#1}{+}}
      {Periodica Mathematica Hungarica}{Period. Math. Hungarica}}
   \ITEE{#2}{PAMS}{\ITE{\equal{#1}{+}}
      {Proceedings of the American Mathematical Society}{Proc. Amer. Math. Soc.}}
   \ITEE{#2}{ProcCambPhS}{\ITE{\equal{#1}{+}}
      {Proceedings of the Cambridge Philosophical Society}{Proc. Cambridge Phil. Soc.}}
   \ITEE{#2}{ProcImpAcadTokyo}{\ITE{\equal{#1}{+}}
      {Proc. Imp. Acad. Tokyo}{Proc. Imp. Acad. Tokyo}}
   \ITEE{#2}{ProcKonink}{\ITE{\equal{#1}{+}}
      {Proceedings of the Koninklijke Nederlandse Akademie van Wetenschappen}{Nederl. Akad. Wetensch. Proc. Ser. A}}
   \ITEE{#2}{PLondMS}{\ITE{\equal{#1}{+}}
      {Proceedings of the London Mathematical Society}{Proc. London Math. Soc.}}
   \ITEE{#2}{PNatlUSA}{\ITE{\equal{#1}{+}}
      {Proceedings of the National Academy of Sciences of the United States of America}{Proc. Natl. Acad. Sci. USA}}
   \ITEE{#2}{PublRIMSKyoto}{\ITE{\equal{#1}{+}}
      {Publ. Res. Inst. Math. Sci. Kyoto Univ.}{Publ. Res. Inst. Math. Sci.}}
   \ITEE{#2}{PWN}{\ITE{\equal{#1}{+}}
      {PWN -- Polish Scientific Publishers, Warszawa}{PWN -- Polish Scientific Publishers, Warszawa}}
   \ITEE{#2}{RCMP}{\ITE{\equal{#1}{+}}
      {Rendiconti del Circolo Matematico di Palermo}{Rend. Circ. Mat. Palermo}}
   \ITEE{#2}{RussMS}{\ITE{\equal{#1}{+}}
      {Russian Mathematical Surveys}{Russian Math. Surveys}}
   \ITEE{#2}{SbM}{\ITE{\equal{#1}{+}}
      {Sbornik: Mathematics}{Sb. Math.}}
   \ITEE{#2}{SciRepTokyoA}{\ITE{\equal{#1}{+}}
      {Science Reports of Tokyo Kyoiku Daigaku, Section A}{Sci. Rep. Tokyo Kyoiku Daigaku Sect. A}}
   \ITEE{#2}{SeminProbStras}{\ITE{\equal{#1}{+}}
      {S\'{e}minaire de probabilit\'{e}s de Strasbourg}{S\'{e}min. Prob. Strasbourg}}
   \ITEE{#2}{SIAMJMAA}{\ITE{\equal{#1}{+}}
      {SIAM Journal on Matrix Analysis and Applications}{SIAM J. Matrix Anal. Appl.}}
   \ITEE{#2}{SibirMZ}{\ITE{\equal{#1}{+}}
      {Sibirski\v{\i} Mat. \v{Z}hurnal}{Sibirsk. Mat. \v{Z}.}}
   \ITEE{#2}{SM}{\ITE{\equal{#1}{+}}
      {Studia Mathematica}{Studia Math.}}
   \ITEE{#2}{TAMS}{\ITE{\equal{#1}{+}}
      {Transactions of the American Mathematical Society}{Trans. Amer. Math. Soc.}}
   \ITEE{#2}{TohokuMJ}{\ITE{\equal{#1}{+}}
      {T\^{o}hoku Mathematical Journal}{T\^{o}hoku Math. J.}}
   \ITEE{#2}{TomskUnivRev}{\ITE{\equal{#1}{+}}
      {Tomsk Universitet Review}{Tomsk. Univ. Rev.}}
   \ITEE{#2}{TopA}{\ITE{\equal{#1}{+}}
      {Topology and its Applications}{Topology Appl.}}
   \ITEE{#2}{TopMethNA}{\ITE{\equal{#1}{+}}
      {Topological Methods in Nonlinear Analysis}{Topol. Methods Nonlinear Anal.}}
   \ITEE{#2}{TsukubaJM}{\ITE{\equal{#1}{+}}
      {Tsukuba Journal of Mathematics}{Tsukuba J. Math.}}
   \ITEE{#2}{UspekhiMN}{\ITE{\equal{#1}{+}}
      {Uspekhi Matem. Nauk}{Uspekhi Mat. Nauk}}
   }

\newcommand{\paplist}[3][]{
   \ITEE{#3}{NIAkhiezer,IMGlazman1993}{
      \BIb{#2}{N.I. Akhiezer and I.M. Glazman}
         {Theory of Linear Operators in Hilbert Space}
         {Dover Publications, Inc., New York}{1993}{#1}}
   \ITEE{#3}{RDAnderson1967}{
      \BIB{#2}{R.D. Anderson}
         {On topological infinite deficiency}
         {\jRN{MichMJ}}{14}{1967}{365--383}{#1}}
   \ITEE{#3}{RDAnderson,JMcCharen1970}{
      \BIB{#2}{R.D. Anderson and J. McCharen}
         {On extending homeomorphisms to Fr\'{e}chet manifolds}
         {\jRN{PAMS}}{25}{1970}{283--289}{#1}}
   \ITEE{#3}{RDAnderson,DWCurtis,JVanMill1982}{
      \BIB{#2}{R.D. Anderson, D.W. Curtis, J. van Mill}
         {A fake topological Hilbert space}
         {\jRN{TAMS}}{272}{1982}{311--321}{#1}}
   \ITEE{#3}{RArens,JEells1956}{
      \BIB{#2}{R. Arens and J. Eells}
         {On embedding uniform and topological spaces}
         {\jRN{PacJM}}{6}{1956}{397--403}{#1}}
   \ITEE{#3}{NAronszajn,PPanitchpakdi1956}{
      \BIB{#2}{N. Aronszajn and P. Panitchpakdi}
         {Extension of uniformly continuous transformations and hyperconvex metric spaces}
         {\jRN{PacJM}}{6}{1956}{405--439}{#1}}
   \ITEE{#3}{KJBabenko1948}{
      \BIB{#2}{K.J. Babenko}
         {On conjugate functions}
         {\jRN{DANSSSR}}{62}{1948}{157--160}{#1}}
   \ITEE{#3}{TBanakh1995}{
      \BIB{#2}{T.O. Banakh}
         {Topology of spaces of probability measures, I}
         {\jRN{MStud}}{5}{1995}{65--87 (Russian)}{#1}}
   \ITEE{#3}{TBanakh1995a}{
      \BIB{#2}{T.O. Banakh}
         {Topology of spaces of probability measures, II}
         {\jRN{MStud}}{5}{1995}{88--106 (Russian)}{#1}}
   \ITEE{#3}{TBanakh1998}{
      \BIB{#2}{T. Banakh}
         {Characterization of spaces admitting a homotopy dense embedding into a Hilbert manifold}
         {\jRN{TopA}}{86}{1998}{123--131}{#1}}
   \ITEE{#3}{TBanakh,TNRadul1997}{
      \BIB{#2}{T.O. Banakh and T.N. Radul}
         {Topology of spaces of probability measures}
         {\jRN{SbM}}{188}{1997}{973--995}{#1}}
   \ITEE{#3}{TBanakh,TRadul,MZarichnyi1996}{
      \BIb{#2}{T. Banakh, T. Radul, M. Zarichnyi}
         {Absorbing sets in infinite-dimensional manifolds}
         {VNTL Publishers, Lviv}{1996}{#1}}
   \ITEE{#3}{TBanakh,IZarichnyy2008}{
      \BIB{#2}{T. Banakh and I. Zarichnyy}
         {Topological groups and convex sets homeomorphic to non-separable Hilbert spaces}
         {\jRN{CEurJM}}{6}{2008}{77--86}{#1}}
   \ITEE{#3}{HBecker,ASKechris1996}{
      \BIb{#2}{H. Becker and A.S. Kechris}
         {The Descriptive Set Theory of Polish Group Actions \textup{(London Math. Soc. Lecture Note Series, vol. 232)}}
         {University Press, Cambridge}{1996}{#1}}
   \ITEE{#3}{GBeer1993}{
      \BIb{#2}{G. Beer}
         {Topologies on Closed and Closed Convex Sets \textup{(Mathematics and Its Applications)}}
         {Kluwer Academic Publishers, Dordrecht}{1993}{#1}}
   \ITEE{#3}{NEBenamara,NNikolski1999}{
      \BIB{#2}{N.E. Benamara and N. Nikolski}
         {Resolvent tests for similarity to a normal operator}
         {\jRN{PLondMS}}{78}{1999}{585--626}{#1}}
   \ITEE{#3}{YBenyamini,JLindenstrauss2000}{
      \BIb{#2}{Y. Benyamini and J. Lindenstrauss}
         {Geometric nonlinear functional analysis I}
         {AMS Colloquium Publications 48}{2000}{#1}}
   \ITEE{#3}{SKBerberian1974}{
      \BIb{#2}{S.K. Berberian}
         {Lectures in Functional Analysis and Operator Theory}
         {Graduate Texts in Mathematics 15, Springer-Verlag, New York}{1974}{#1}}
   \ITEE{#3}{SNBernstein1954}{
      \BIb{#2}{S.N. Bernstein}
         {Collected Works II}
         {Akad. Nauk SSSR, Moscow}{1954 (Russian)}{#1}}
   \ITEE{#3}{CzBessaga,APelczynski1972}{
      \BIB{#2}{Cz. Bessaga and A. Pe\l{}czy\'{n}ski}
         {On spaces of measurable functions}
         {\jRN{SM}}{44}{1972}{597--615}{#1}}
   \ITEE{#3}{CzBessaga,APelczynski1975}{
      \BIb{#2}{Cz. Bessaga and A. Pe\l{}czy\'{n}ski}
         {Selected topics in infinite-dimensional topology}
         {\jRN{PWN}}{1975}{#1}}
   \ITEE{#3}{MBestvina,JMogilski1986}{
      \BIB{#2}{M. Bestvina and J. Mogilski}
         {Characterizing certain incomplete infinite-dimensional absolute retracts}
         {\jRN{MichMJ}}{33}{1986}{291--313}{#1}}
   \ITEE{#3}{MBestvina,PBowers,JMogilsky,JWalsh1986}{
      \BIB{#2}{M. Bestvina, P. Bowers, J. Mogilsky, J. Walsh}
         {Characterization of Hilbert space manifolds revisited}
         {\jRN{TopA}}{24}{1986}{53--69}{#1}}
   \ITEE{#3}{RBhatia1997}{
      \BIb{#2}{R. Bhatia}
         {Matrix Analysis}
         {Springer, New York}{1997}{#1}}
   \ITEE{#3}{GBirkhoff1936}{
      \BIB{#2}{G. Birkhoff}
         {A note on topological groups}
         {\jRN{ComposM}}{3}{1936}{427--430}{#1}}
   \ITEE{#3}{MSBirman,MZSolomjak1987}{
      \BIb{#2}{M.S. Birman and M.Z. Solomjak}
         {Spectral Theory of Self-Adjoint Operators in Hilbert Space}
         {D. Reidel Publishing Co., Dordrecht}{1987}{#1}}
   \ITEE{#3}{EBishop1961}{
      \BIB{#2}{E. Bishop}
         {A generalization of the Stone-Weierstrass theorem}
         {\jRN{PacJM}}{11}{1961}{777--783}{#1}}
   \ITEE{#3}{BBlackadar2006}{
      \BIb{#2}{B. Blackadar}{Operator Algebras. Theory of $\CCc^*$-algebras and von Neumann algebras 
         \textup{(Encyclopaedia of Mathematical Sciences, vol. 122: Operator Algebras and Non-Commutative Geometry III)}}
         {Springer-Verlag, Berlin-Heidelberg}{2006}{#1}}
   \ITEE{#3}{JBlass,WHolsztynski1972}{
      \BIB{#2}{J. Blass and W. Holszty\'{n}ski}
         {Cubical polyhedra and homotopy III}
         {\jRN{AttiAccLincRendNat}}{53}{1972}{275--279}{#1}}
   \ITEE{#3}{FFBonsall,NJDuncan1973}{
      \BIb{#2}{F.F. Bonsall and N.J. Duncan}
         {Complete Normed Algebras}
         {Springer Verlag, Berlin}{1973}{#1}}
   \ITEE{#3}{NBourbaki2002}{
      \BIb{#2}{N. Bourbaki}
         {Lie Groups and Lie Algebras, Chapters 4--6}
         {Springer, New York}{2002}{#1}}
   \ITEE{#3}{PLBowers1989}{
      \BIB{#2}{P.L. Bowers}
         {Limitation topologies on function spaces}
         {\jRN{TAMS}}{314}{1989}{421--431}{#1}}
   \ITEE{#3}{ABrown1953}{
      \BIB{#2}{A. Brown}
         {On a class of operators}
         {\jRN{PAMS}}{4}{1953}{723--728}{#1}}
   \ITEE{#3}{ABrown,CKFong,DWHadwin1978}{
      \BIB{#2}{A. Brown, C.-K. Fong, D.W. Hadwin}
         {Parts of operators on Hilbert space}
         {\jRN{IllinoisJM}}{22}{1978}{306--314}{#1}}
   \ITEE{#3}{AMBruckner,JBBruckner,BSThomson1997}{
      \BIb{#2}{A.M. Bruckner, J.B. Bruckner, B.S. Thomson}
         {Real Analysis}
         {Prentice-Hall, New Jersey}{1997}{#1}}
   \ITEE{#3}{PJCameron,AMVershik2006}{
      \BIB{#2}{P.J. Cameron and A.M. Vershik}
         {Some isometry groups of Urysohn space}
         {\jRN{AnnPALog}}{143}{2006}{70--78}{#1}}
   \ITEE{#3}{CCastaing1966}{
      \BIB{#2}{C. Castaing}
         {Quelques probl\`{e}mes de mesurabilit\'{e} li\'{e}es \`{a} la th\'{e}orie de la commande}
         {\jRN{CRParis}}{262}{1966}{409--411}{#1}}
   \ITEE{#3}{JAVanCasteren1980}{
      \BIB{#2}{J.A. van Casteren}
         {A problem of Sz.-Nagy}
         {\jRN{ActaSM}}{42}{1980}{189--194}{#1}}
   \ITEE{#3}{JAVanCasteren1983}{
      \BIB{#2}{J.A. van Casteren}
         {Operators similar to unitary or selfadjoint ones}
         {\jRN{PacJM}}{104}{1983}{241--255}{#1}}
   \ITEE{#3}{XCatepillan,MPtak,WSzymanski1994}{
      \BIB{#2}{X. Catepill\'{a}n, M. Ptak, W. Szyma\'{n}ski}
         {Multiple canonical decompositions of families of operators and a model of quasinormal families}
         {\jRN{PAMS}}{121}{1994}{1165--1172}{#1}}
   \ITEE{#3}{RCauty1994}{
      \BIB{#2}{R. Cauty}
         {Un espace m\'{e}trique lin\'{e}aire qui n'est pas un r\'{e}tracte absolu}
         {\jRN{FM}}{146}{1994}{85--99, (French)}{#1}}
   \ITEE{#3}{TAChapman1971}{
      \BIB{#2}{T.A. Chapman}
         {Deficiency in infinite-dimensional manifolds}
         {\jRN{GTopA}}{1}{1971}{263--272}{#1}}
   \ITEE{#3}{TAChapman1976}{
      \BIb{#2}{T.A. Chapman}
         {Lectures on Hilbert cube manifolds}
         {C.B.M.S. Regional Conference Series in Math. No 28, Amer. Math. Soc.}{1976}{#1}}
   \ITEE{#3}{RBChuaqui1977}{
      \BIB{#2}{R.B. Chuaqui}
         {Measures invariant under a group of transformations}
         {\jRN{PacJM}}{68}{1977}{313--329}{#1}}
   \ITEE{#3}{JBConway1985}{
      \BIb{#2}{J.B. Conway}
         {A Course in Functional Analysis}
         {Springer-Verlag, New York}{1985}{#1}}
   \ITEE{#3}{JBConway2000}{
      \BIb{#2}{J.B. Conway}
         {A Course in Operator Theory}
         {(Graduate Studies in Mathematics, vol. 21) Amer. Math. Soc., Providence}{2000}{#1}}
   \ITEE{#3}{GCorach,AMaestripieri,MMbekhta2009}{
      \BIB{#2}{G. Corach, A. Maestripieri, M. Mbekhta}
         {Metric and homogeneous structure of closed range operators}
         {\jRN{JOT}}{61}{2009}{171--190}{#1}}
   \ITEE{#3}{MJCowen,RGDouglas1978}{
      \BIB{#2}{M.J. Cowen and R.G. Douglas}
         {Complex geometry and operator theory}
         {\jRN{ActaM}}{141}{1978}{187--261}{#1}}
   \ITEE{#3}{DWCurtis1985}{
      \BIB{#2}{D.W. Curtis}
         {Boundary sets in the Hilbert cube}
         {\jRN{TopA}}{20}{1985}{201--221}{#1}}
   \ITEE{#3}{MMDay1958}{
      \BIb{#2}{M.M. Day}
         {Normed Linear Spaces}
         {Springer Verlag, Berlin}{1958}{#1}}
   \ITEE{#3}{CDellacherie1967}{
      \BIB{#2}{C. Dellacherie}
         {Un compl\'{e}ment au th\'{e}or\`{e}me de Weierstrass-Stone}
         {\jRN{SeminProbStras}}{1}{1967}{52--53}{#1}}
   \ITEE{#3}{JJDijkstra1987}{
      \BIB{#2}{J.J. Dijkstra}
         {Strong negligibility of $\sigma$-compacta does not characterize Hilbert space}
         {\jRN{PacJM}}{127}{1987}{19--30}{#1}}
   \ITEE{#3}{JJDijkstra1990}{
      \BIB{#2}{J.J. Dijkstra}
         {Characterizing Hilbert space topology in terms of strong negligibility}
         {\jRN{ComposM}}{75}{1990}{299--306}{#1}}
   \ITEE{#3}{TDobrowolski,WMarciszewski2002}{
      \BIB{#2}{T. Dobrowolski and W. Marciszewski}
         {Failure of the Factor Theorem for Borel pre-Hilbert spaces}
         {\jRN{FM}}{175}{2002}{53--68}{#1}}
   \ITEE{#3}{TDobrowolski,JMogilski1990}{
      \BiB{#2}{T. Dobrowolski and J. Mogilski}
         {Problems on Topological Classification of Incomplete Metric Spaces}{Chapter 25 in:}
         {Open Problems in Topology}{J. van Mill and G.M. Reed (eds.), North-Holland Amsterdam}{1990}{411--429}{#1}}
   \ITEE{#3}{TDobrowolski,HTorunczyk1981}{
      \BIB{#2}{T. Dobrowolski and H. Toru\'{n}czyk}
         {Separable complete ANR's admitting a group structure are Hilbert manifolds}
         {\jRN{TopA}}{12}{1981}{229--235}{#1}}
   \ITEE{#3}{RGDouglas1966}{
      \BIB{#2}{R.G. Douglas}
         {On majorization, factorization and range inclusion of operators in Hilbert space}
         {\jRN{PAMS}}{17}{1966}{413--416}{#1}}
   \ITEE{#3}{CHDowker1947}{
      \BIB{#2}{C.H. Dowker}
         {Mapping theorems for non-compact spaces}
         {\jRN{AmJM}}{69}{1947}{200--242}{#1}}
   \ITEE{#3}{CHDowker1952}{
      \BIB{#2}{C.H. Dowker}
         {Topology of metric complexes}
         {\jRN{AmJM}}{74}{1952}{555--577}{#1}}
   \ITEE{#3}{JDugundji1951}{
      \BIB{#2}{J. Dugundji}
         {An extension of Tietze's theorem}
         {\jRN{PacJM}}{1}{1951}{353--367}{#1}}
   \ITEE{#3}{JDugundji1958}{
      \BIB{#2}{J. Dugundji}
         {Absolute neighborhood retracts and local connectedness for arbitrary metric spaces}
         {\jRN{ComposM}}{13}{1958}{229--246}{#1}}
   \ITEE{#3}{JDugundji1965}{
      \BIB{#2}{J. Dugundji}
         {Locally equiconnected spaces and absolute neighborhood retracts}
         {\jRN{FM}}{57}{1965}{187--193}{#1}}
   \ITEE{#3}{NDunford,JTSchwartz1958}{
      \BIb{#2}{N. Dunford and J.T. Schwartz}
         {Linear Operators, part I}
         {Interscience Publishers, New York}{1958}{#1}}
   \ITEE{#3}{NDunford,JTSchwartz1963}{
      \BIb{#2}{N. Dunford and J.T. Schwartz}
         {Linear Operators, part II}
         {Interscience Publishers, New York}{1963}{#1}}
   \ITEE{#3}{NDunford,JTSchwartz1971}{
      \BIb{#2}{N. Dunford and J.T. Schwartz}
         {Linear Operators, part III}
         {Wiley-Interscience, New York}{1971}{#1}}
   \ITEE{#3}{MLEaton,MDPerlman1977}{
      \BIB{#2}{M.L. Eaton and M.D. Perlman}
         {Reflection groups, generalized Schur functions and the geometry of majorization}
         {\jRN{AnnProb}}{5}{1977}{829--860}{#1}}
   \ITEE{#3}{EGEffros1965}{
      \BIB{#2}{E.G. Effros}
         {The Borel space of von Neumann algebras on a separable Hilbert space}
         {\jRN{PacJM}}{15}{1965}{1153--1164}{#1}}
   \ITEE{#3}{EGEffros1966}{
      \BIB{#2}{E.G. Effros}
         {Global structure in von Neumann algebras}
         {\jRN{TAMS}}{121}{1966}{434--454}{#1}}
   \ITEE{#3}{REspinola,MAKhamsi2001}{
      \BiB{#2}{R. Espinola and M.A. Khamsi}
         {Introduction to hyperconvex spaces}{Chapter XIII in:}{Handbook of Metric Fixed Point Theory}
         {W.A. Kirk and B. Sims (editors), Kluwer Academic Publishers}{2001}{391--435}{#1}}
   \ITEE{#3}{PAFillmore,JPWilliams1971}{
      \BIB{#2}{P.A. Fillmore and J.P. Williams}
         {On operator ranges}
         {\jRN{AdvM}}{7}{1971}{254--281}{#1}}
   \ITEE{#3}{JEells,NHKuiper1969}{
      \BIB{#2}{J. Eells and N.H. Kuiper}
         {Homotopy negligible subsets in infinite-dimensional manifolds}
         {\jRN{ComposM}}{21}{1969}{151--161}{#1}}
   \ITEE{#3}{REngelking1977}{
      \BIb{#2}{R. Engelking}
         {General Topology}
         {\jRN{PWN}}{1977}{#1}}
   \ITEE{#3}{REngelking1978}{
      \BIb{#2}{R. Engelking}
         {Dimension Theory}
         {\jRN{PWN}}{1978}{#1}}
   \ITEE{#3}{REngelking1989}{
      \BIb{#2}{R. Engelking}
         {General Topology. Revised and completed edition \textup{(Sigma series in pure mathematics, vol. 6)}}
         {Heldermann Verlag, Berlin}{1989}{#1}}
   \ITEE{#3}{PErdos,RDMauldin1976}{
      \BIB{#2}{P. Erd\"{o}s and R.D. Mauldin}
         {The nonexistence of certain invariant measures}
         {\jRN{PAMS}}{59}{1976}{321--322}{#1}}
   \ITEE{#3}{JErnest1976}{
      \BIB{#2}{J. Ernest}
         {Charting the operator terrain}
         {\jRN{MAMS}}{171}{1976}{207 pp}{#1}}
   \ITEE{#3}{RHFox1943}{
      \BIB{#2}{R.H. Fox}
         {On fiber spaces, II}
         {\jRN{BAMS}}{49}{1943}{733--735}{#1}}
   \ITEE{#3}{NAFriedman1970}{
      \BIb{#2}{N.A. Friedman}
         {Introduction to ergodic theory}
         {Van Nostrand Reinhold Company}{1970}{#1}}
   \ITEE{#3}{MFujii,MKajiwara,YKato,FKubo1976}{
      \BIB{#2}{M. Fujii, M. Kajiwara, Y. Kato, F. Kubo}
         {Decompositions of operators in Hilbert spaces}
         {\jRN{MathJap}}{21}{1976}{117--120}{#1}}
   \ITEE{#3}{SGao,ASKechris2003}{
      \BIB{#2}{S. Gao and A.S. Kechris}
         {On the classification of Polish metric spaces up to isometry}
         {\jRN{MAMS}}{161}{2003}{viii+78}{#1}}
   \ITEE{#3}{MIGarrido,FMontalvo1991}{
      \BIB{#2}{M.I. Garrido and F. Montalvo}
         {On some generalizations of the Kakutani-Stone and Stone-Weierstrass theorems}
         {\jRN{ExtrM}}{6}{1991}{156--159}{#1}}
   \ITEE{#3}{LGe,JShen2002}{
      \BIB{#2}{L. Ge and J. Shen}
         {Generator problem for certain property T factors}
         {\jRN{PNAS}}{99}{2002}{565--567}{#1}}
   \ITEE{#3}{IMGelfand,MANaimark1943}{
      \BIB{#2}{I.M. Gelfand and M.A. Naimark}
         {On the embedding of normed rings into the ring of operators in Hilbert space}
         {\jRN{MSb}}{12}{1943}{197--213}{#1}}
   \ITEE{#3}{FGesztesy,MMalamud,MMitrea,SNaboko2009}{
      \BIB{#2}{F. Gesztesy, M. Malamud, M. Mitrea, S. Naboko}
         {Generalized polar decompositions for closed operators in Hilbert spaces and some applications}
         {\jRN{IEOT}}{64}{2009}{83--113}{#1}}
   \ITEE{#3}{LGillman,MJerison1960}{
      \BIb{#2}{L. Gillman and M. Jerison}
         {Rings of continuous functions}
         {New York}{1960}{#1}}
   \ITEE{#3}{JGlimm1960}{
      \BIB{#2}{J. Glimm}
         {A Stone-Weierstrass theorem for $\CCc^*$-algebras}
         {\jRN{AnnM}}{72}{1960}{216--244}{#1}}
   \ITEE{#3}{GGodefroy,NJKalton2003}{
      \BIB{#2}{G. Godefroy and N.J. Kalton}
         {Lipschitz-free Banach spaces}
         {\jRN{SM}}{159}{2003}{121--141}{#1}}
   \ITEE{#3}{ICGohberg,MGKrein1967}{
      \BIB{#2}{I.C. Gohberg and M.G. Krein}
         {On a description of contraction operators similar to unitary ones}
         {\jRN{FunkAnalPril}}{1}{1967}{38--60}{#1}}
   \ITEE{#3}{ELGriffinJr1953}{
      \BIB{#2}{E.L. Griffin Jr.}
         {Some contributions to the theory of rings of operators}
         {\jRN{TAMS}}{75}{1953}{471--504}{#1}}
   \ITEE{#3}{ELGriffinJr1955}{
      \BIB{#2}{E.L. Griffin Jr.}
         {Some contributions to the theory of rings of operators II}
         {\jRN{TAMS}}{79}{1955}{389--400}{#1}}
   \ITEE{#3}{MGromov1981}{
      \BIB{#2}{M. Gromov}
         {Groups of polynomial growth and expanding maps}
         {\jRN{InHauEtSPM}}{53}{1981}{53--73}{#1}}
   \ITEE{#3}{MGromov1999}{
      \BIb{#2}{M. Gromov}
         {Metric Structures for Riemannian and Non-Riemannian Spaces}
         {Progress in Math. \textbf{152}, Birkh\"{a}user}{1999}{#1}}
   \ITEE{#3}{JDeGroot1956}{
      \BIB{#2}{J. de Groot}
         {Non-archimedean metrics in topology}
         {\jRN{PAMS}}{7}{1956}{948--953}{#1}}
   \ITEE{#3}{LCGrove,CTBenson1985}{
      \BIb{#2}{L.C. Grove and C.T. Benson}
         {Finite Reflection Group}
         {2nd ed., Springer-Verlag}{1985}{#1}}
   \ITEE{#3}{VIGurarii1966}{
      \BIB{#2}{V.I. Gurari\v{\i}}{Spaces of universal placement, isotropic spaces and a problem of Mazur 
         on rotations of Banach spaces \textup{(Russian)}}
         {\jRN{SibirMZ}}{7}{1966}{1002--1013}{#1}}
   \ITEE{#3}{DWHadwin1976}{
      \BIB{#2}{D.W. Hadwin}
         {An operator-valued spectrum}
         {\jRN{NAMS}}{23}{1976}{A-163}{#1}}
   \ITEE{#3}{DWHadwin1977}{
      \BIB{#2}{D.W. Hadwin}
         {An operator-valued spectrum}
         {\jRN{IndianaUMJ}}{26}{1977}{329--340}{#1}}
   \ITEE{#3}{HHahn1932}{
      \BIb{#2}{H. Hahn}
         {Reelle Funktionen I}
         {Leipzig}{1932}{#1}}
   \ITEE{#3}{PRHalmos1950}{
      \BIb{#2}{P.R. Halmos}
         {Measure theory}
         {Van Nostrand, New York}{1950}{#1}}
   \ITEE{#3}{PRHalmos1951}{
      \BIb{#2}{P.R. Halmos}
         {Introduction to Hilbert Space and the Theory of Spectral Multiplicity}
         {Chelsea Publishing Company, New York}{1951}{#1}}
   \ITEE{#3}{PRHalmos1956}{
      \BIb{#2}{P.R. Halmos}
         {Lectures on Ergodic Theory}
         {Publ. Math. Soc. Japan, Tokyo}{1956}{#1}}
   \ITEE{#3}{PRHalmos1982}{
      \BIb{#2}{P.R. Halmos}
         {A Hilbert Space Problem Book}
         {Springer-Verlag New York Inc.}{1982}{#1}}
  \ITEE{#3}{PRHalmos,JEMcLaughlin1963}{
      \BIB{#2}{P.R. Halmos and J.E. McLaughlin}
         {Partial isometries}
         {\jRN{PacJM}}{13}{1963}{585--596}{#1}}
   \ITEE{#3}{RWHansell1972}{
      \BIB{#2}{R.W. Hansell}
         {On the nonseparable theory of Borel and Souslin sets}
         {\jRN{BAMS}}{78}{1972}{236--241}{#1}}
   \ITEE{#3}{FHausdorff1930}{
      \BIB{#2}{F. Hausdorff}
         {Erweiterung einer Hom\"{o}omorphie}
         {\jRN{FM}}{16}{1930}{353--360}{#1}}
   \ITEE{#3}{FHausdorff1934}{
      \BIB{#2}{F. Hausdorff}
         {\"{U}ber innere Abbildungen}
         {\jRN{FM}}{23}{1934}{279--291}{#1}}
   \ITEE{#3}{FHausdorff1938}{
      \BIB{#2}{F. Hausdorff}
         {Erweiterung einer stetigen Abbildung}
         {\jRN{FM}}{30}{1938}{40--47}{#1}}
   \ITEE{#3}{DWHenderson1971}{
      \BIB{#2}{D.W. Henderson}
         {Corrections and extensions of two papers about infinite-dimensional manifolds}
         {\jRN{GTopA}}{1}{1971}{321--327}{#1}}
   \ITEE{#3}{DWHenderson1975}{
      \BIB{#2}{D.W. Henderson}
         {$Z$-sets in ANR's}
         {\jRN{TAMS}}{213}{1975}{205--216}{#1}}
   \ITEE{#3}{DWHenderson,RMSchori1970}{
      \BIB{#2}{D.W. Henderson and R.M. Schori}
         {Topological classification of infinite-dimensional manifolds by homotopy type}
         {\jRN{BAMS}}{76}{1970}{121--124}{#1}}
   \ITEE{#3}{DWHenderson,JEWest1970}{
      \BIB{#2}{D.W. Henderson and J.E. West}
         {Triangulated infinite-dimensional manifolds}
         {\jRN{BAMS}}{76}{1970}{655--660}{#1}}
   \ITEE{#3}{BHoffmann1979}{
      \BIB{#2}{B. Hoffmann}
         {A compact contractible topological group is trivial}
         {\jRN{ArchM}}{32}{1979}{585--587}{#1}}
   \ITEE{#3}{DHofmann2002}{
      \BIB{#2}{D. Hofmann}
         {On a generalization of the Stone-Weierstrass theorem}
         {\jRN{ACS}}{10}{2002}{569--592}{#1}}
   \ITEE{#3}{GHognas,AMukherjea1995}{
      \BIb{#2}{G. H\"ogn\"as and A. Mukherjea}
         {Probability Measures on Semigroups. Convolution Products, Random Walks, and Random Matrices}
         {Plenum Press, New York}{1995}{#1}}
   \ITEE{#3}{MRHolmes1992}{
      \BIB{#2}{M.R. Holmes}
         {The universal separable metric space of Urysohn and isometric embeddings thereof in Banach spaces}
         {\jRN{FM}}{140}{1992}{199--223}{#1}}
   \ITEE{#3}{MRHolmes2008}{
      \BIB{#2}{M.R. Holmes}
         {The Urysohn space embeds in Banach spaces in just one way}
         {\jRN{TopA}}{155}{2008}{1479--1482}{#1}}
   \ITEE{#3}{RRHolmes,TYTam1999}{
      \BIB{#2}{R.R. Holmes and T.Y. Tam}
         {Distance to the convex hull of an orbit under the action of a compact group}
         {\jRN{JAusMSA}}{66}{1999}{331--357}{#1}}
   \ITEE{#3}{RHorn,RMathias1990}{
      \BIB{#2}{R. Horn and R. Mathias}
         {Cauchy-Schwartz inequalities associated with positive semidefinite matrices}
         {\jRN{LAA}}{142}{1990}{63--82}{#1}}
   \ITEE{#3}{GEHuhunaisvili1955}{
      \BIB{#2}{G.E. Huhunai\v{s}vili}
         {On a property of Urysohn's universal metric space}
         {\jRN{DANSSSR}}{101}{1955}{607--610 (Russian)}{#1}}
   \ITEE{#3}{JEHumphreys1990}{
      \BIb{#2}{J.E. Humphreys}
         {Reflection Groups and Coxeter Groups}
         {Cambridge University Press}{1990}{#1}}
   \ITEE{#3}{JRIsbell1964}{
      \BIB{#2}{J.R. Isbell}
         {Six theorems about injective metric spaces}
         {\jRN{CMHelv}}{39}{1964}{65--76}{#1}}
   \ITEE{#3}{SIzumino,YKato1985}{
      \BIB{#2}{S. Izumino and Y. Kato}
         {The closure of invertible operators on Hilbert space}
         {\jRN{ActaSM}}{49}{1985}{321--327}{#1}}
   \ITEE{#3}{CJiang2004}{
      \BIB{#2}{C. Jiang}
         {Similarity classification of Cowen-Douglas operators}
         {\jRN{CanadJM}}{56}{2004}{742--775}{#1}}
   \ITEE{#3}{WBJohnson,JLindenstrauss2001}{
      \BiB{#2}{W.B. Johnson and J. Lindenstrauss}{Basic Concepts in the Geometry of Banach Spaces}
         {Chapter 1 in:}{Handbook of the Geometry of Banach Spaces, Vol. 1}
         {W.B. Johnson and J. Lindenstrauss (editors), Elsevier Science B.V., Amsterdam}{2001}{1--84}{#1}}
   \ITEE{#3}{IBJung,JStochel2008}{
      \BIB{#2}{I.B. Jung and J. Stochel}
         {Subnormal operators whose adjoints have rich point spectrum}
         {\jRN{JFA}}{255}{2008}{1797--1816}{#1}}
   \ITEE{#3}{RVKadison,JRRingrose1983}{
      \BIb{#2}{R.V. Kadison and J.R. Ringrose}
         {Fundamentals of the Theory of Operator Algebras. Volume I: Elementary Theory}
         {Academic Press, Inc., New York-London}{1983}{#1}}
   \ITEE{#3}{RVKadison,JRRingrose1986}{
      \BIb{#2}{R.V. Kadison and J.R. Ringrose}
         {Fundamentals of the Theory of Operator Algebras. Volume II: Advanced Theory}
         {Academic Press, Inc., Orlando-London}{1986}{#1}}
   \ITEE{#3}{SKakutani1936}{
      \BIB{#2}{S. Kakutani}
         {\"{U}ber die Metrisation der topologischen Gruppen}
         {\jRN{ProcImpAcadTokyo}}{12}{1936}{82--84}{#1}}
   \ITEE{#3}{SKakutani1938}{
      \BIB{#2}{S. Kakutani}
         {Two fixed-point theorems concerning bicompact convex sets}
         {\jRN{ProcImpAcadTokyo}}{14}{1938}{242--245}{#1}}
   \ITEE{#3}{SKakutani1941}{
      \BIB{#2}{S. Kakutani}
         {Concrete representation of abstract L-spaces}
         {\jRN{AnnM}}{42}{1941}{523--537}{#1}}
   \ITEE{#3}{SKakutani1941a}{
      \BIB{#2}{S. Kakutani}
         {Concrete representation of abstract M-spaces}
         {\jRN{AnnM}}{42}{1941}{994--1024}{#1}}
   \ITEE{#3}{NKalton2007}{
      \BIB{#2}{N. Kalton}
         {Extending Lipschitz maps into $\CCc(K)$-spaces}
         {\jRN{IsraelJM}}{162}{2007}{275--315}{#1}}
   \ITEE{#3}{RKane2001}{
      \BIb{#2}{R. Kane}
         {Reflection Groups and Invariant Theory}
         {Canadian Mathematical Society, Springer}{2001}{#1}}
   \ITEE{#3}{VKannan,SRRaju1980}{
      \BIB{#2}{V. Kannan and S.R. Raju}
         {The nonexistence of invariant universal measures on semigroups}
         {\jRN{PAMS}}{78}{1980}{482--484}{#1}}
   \ITEE{#3}{IKaplansky1951}{
      \BIB{#2}{I. Kaplansky}
         {A theorem on rings of operators}
         {\jRN{PacJM}}{1}{1951}{227--232}{#1}}
   \ITEE{#3}{MKatetov1988}{
      \BiB{#2}{M. Kat\v{e}tov}{On universal metric spaces}{in: Frolik (ed.),}
         {General Topology and its Relations to Modern Analysis and Algebra VI. Proceedings of the Sixth Prague 
         Topological Symposium 1986}{Heldermann Verlag Berlin}{1988}{323--330}{#1}}
   \ITEE{#3}{YKatznelson1960}{
      \BIB{#2}{Y. Katznelson}
         {Sur les alg\'{e}bres dont les \'{e}l\'{e}ments non n\'{e}gatifs admettent des racines carr\'{e}es}
         {\jRN{AnnSciEcNormSupT}}{77}{1960}{167--174}{#1}}
   \ITEE{#3}{OHKeller1931}{
      \BIB{#2}{O.H. Keller}
         {Die Homoiomorphie der kompakten konvexen Mengen in Hilbertschen Raum}
         {\jRN{MAnn}}{105}{1931}{748--758}{#1}}
   \ITEE{#3}{MAKhamsi,WAKirk,CMartinez2000}{
      \BIB{#2}{M.A. Khamsi, W.A. Kirk, C. Martinez}
         {Fixed point and selection theorems in hyperconvex spaces}
         {\jRN{PAMS}}{128}{2000}{3275--3283}{#1}}
   \ITEE{#3}{ABKhararazishvili1998}{
      \BIb{#2}{A.B. Khararazishvili}
         {Transformation groups and invariant measures. Set-theoretic aspects}
         {World Scientific Publishing Co., Inc., River Edge, NJ}{1998}{#1}}
   \ITEE{#3}{YKijima1987}{
      \BIB{#2}{Y. Kijima}
         {Fixed points of nonexpansive self-maps of a compact metric space}
         {\jRN{JMAnApp}}{123}{1987}{114--116}{#1}}
  \ITEE{#3}{JSKim,ChRKim,SGLee1980}{
      \BIB{#2}{J.S. Kim, Ch.R. Kim, S.G. Lee}
         {Reducing operator valued spectra of a Hilbert space operator}
         {\jRN{JKoreanMS}}{17}{1980}{123--129}{#1}}
   \ITEE{#3}{JKindler1995}{
      \BIB{#2}{J. Kindler}
         {Minimax theorems with applications to convex metric spaces}
         {\jRN{CollM}}{68}{1995}{179--186}{#1}}
   \ITEE{#3}{WAKirk1998}{
      \BIB{#2}{W.A. Kirk}
         {Hyperconvexity of $\RRR$-trees}
         {\jRN{FM}}{156}{1998}{67--72}{#1}}
   \ITEE{#3}{VLKleeJr1952}{
      \BIB{#2}{V.L. Klee Jr.}
         {Invariant metrics in groups (solution of a problem of Banach)}
         {\jRN{PAMS}}{3}{1952}{484--487}{#1}}
   \ITEE{#3}{HJKowalsky1957}{
      \BIB{#2}{H.J. Kowalsky}
         {Einbettung metrischer R\"{a}ume}
         {\jRN{ArchM}}{8}{1957}{336--339}{#1}}
   \ITEE{#3}{WKubis,MRubin2010}{
      \BIB{#2}{W. Kubi\'{s} and M. Rubin}
         {Extension and reconstruction theorems for the Urysohn universal metric space}
         {\jRN{CzMJ}}{60}{2010}{1--29}{#1}}
   \ITEE{#3}{KKuratowski1966}{
      \BIb{#2}{K. Kuratowski}
         {Topology. \textup{Vol. I}}
         {\jRN{PWN}}{1966}{#1}}
   \ITEE{#3}{KKuratowski,BKnaster1927}{
      \BIB{#2}{K. Kuratowski and B. Knaster}
         {A connected and connected im kleinen point set which contains no perfect subset}
         {\jRN{BAMS}}{33}{1927}{106--109}{#1}}
   \ITEE{#3}{KKuratowski,AMostowski1976}{
      \BIb{#2}{K. Kuratowski and A. Mostowski}
         {Set Theory with an Introduction to Descriptive Set Theory}
         {\jRN{PWN}}{1976}{#1}}
   \ITEE{#3}{GLewicki1992}{
      \BIB{#2}{G. Lewicki}
         {Bernstein's ``lethargy'' theorem in metrizable topological linear spaces}
         {\jRN{MonatM}}{113}{1992}{213--226}{#1}}
   \ITEE{#3}{ASLewis1996}{
      \BIB{#2}{A.S. Lewis}
         {Group invariance and convex matrix analysis}
         {\jRN{SIAMJMAA}}{17}{1996}{927--949}{#1}}
   \ITEE{#3}{C-KLi,N-KTsing1991}{
      \BIB{#2}{C.-K. Li and N.-K. Tsing}
         {$G$-invariant norms and $G(c)$-radii}
         {\jRN{LAA}}{150}{1991}{179--194}{#1}}
   \ITEE{#3}{AJLazar,JLindenstrauss1971}{
      \BIB{#2}{A.J. Lazar and J. Lindenstrauss}
         {Banach spaces whose duals are $L_1$ spaces and their representing matrices}
         {\jRN{ActaM}}{126}{1971}{165--193}{#1}}
   \ITEE{#3}{EHLieb,MLoss1997}{
      \BIb{#2}{E.H. Lieb and M. Loss}
         {Analysis \textup{(Graduate Studies in Mathematics, vol. 14)}}
         {Amer. Math. Soc., Providence, RI}{1997}{#1}}
   \ITEE{#3}{ALindenbaum1926}{
      \BIB{#2}{A. Lindenbaum}
         {Contributions \`{a} l'\'{e}tude de l'espace m\'{e}trique I}
         {\jRN{FM}}{8}{1926}{209--222}{#1}}
   \ITEE{#3}{DLindenstrauss,LTzafriri1971}{
      \BIB{#2}{D. Lindenstrauss and L. Tzafriri}
         {On the complemented subspaces problem}
         {\jRN{IsraelJM}}{9}{1971}{263--269}{#1}}
   \ITEE{#3}{RILoebl1986}{
      \BIB{#2}{R.I. Loebl}
         {A note on containment of operators}
         {\jRN{BAustrMS}}{33}{1986}{279--291}{#1}}
   \ITEE{#3}{LHLoomis1945}{
      \BIB{#2}{L.H. Loomis}
         {Abstract congruence and the uniqueness of Haar measure}
         {\jRN{AnnM}}{46}{1945}{348--355}{#1}}
   \ITEE{#3}{LHLoomis1949}{
      \BIB{#2}{L.H. Loomis}
         {Haar measure in uniform structures}
         {\jRN{DukeMJ}}{16}{1949}{193--208}{#1}}
   \ITEE{#3}{ERLorch1939}{
      \BIB{#2}{E.R. Lorch}
         {Bicontinuous linear transformation in certain vector spaces}
         {\jRN{BAMS}}{45}{1939}{564--569}{#1}}
   \ITEE{#3}{ATLundell,SWeingram1969}{
      \BIb{#2}{A.T. Lundell and S. Weingram}
         {The topology of CW-complexes}
         {Litton Educ. Publ.}{1969}{#1}}
   \ITEE{#3}{WLusky1976}{
      \BIB{#2}{W. Lusky}
         {The Gurarij spaces are unique}
         {\jRN{ArchM}}{27}{1976}{627--635}{#1}}
   \ITEE{#3}{WLusky1977}{
      \BIB{#2}{W. Lusky}
         {On separable Lindenstrauss spaces}
         {\jRN{JFA}}{26}{1977}{103--120}{#1}}
   \ITEE{#3}{DMaharam1942}{
      \BIB{#2}{D. Maharam}
         {On homogeneous measure algebras}
         {\jRN{PNatlUSA}}{28}{1942}{108--111}{#1}}
   \ITEE{#3}{MMalicki,SSolecki2009}{
      \BIB{#2}{M. Malicki and S. Solecki}
         {Isometry groups of separable metric spaces}
         {\jRN{MProcCambPhS}}{146}{2009}{67--81}{#1}}
   \ITEE{#3}{PMankiewicz1972}{
      \BIB{#2}{P. Mankiewicz}
         {On extension of isometries in normed linear spaces}
         {\jRN{BAPolSSSM}}{20}{1972}{367--371}{#1}}
   \ITEE{#3}{JMartinezMaurica,MTPellon1987}{
      \BIB{#2}{J. Martinez-Maurica and M.T. Pell\'{o}n}
         {Non-archimedean Chebyshev centers}
         {\jRN{IndagMP}}{90}{1987}{417--421}{#1}}
   \ITEE{#3}{KMaurin1980}{
      \BIb{#2}{K. Maurin}
         {Analysis, Part II}
         {D. Reidel, Dordrecht-Boston-London}{1980}{#1}}
   \ITEE{#3}{SMazur,SUlam1932}{
      \BIB{#2}{S. Mazur and S. Ulam}
         {Sur les transformationes isom\'{e}triques d'espaces vectoriels norm\'{e}s}
         {\jRN{CRASParis}}{194}{1932}{946--948}{#1}}
   \ITEE{#3}{SMazurkiewicz1920}{
      \BIB{#2}{S. Mazurkiewicz}
         {Sur les lignes de Jordan}
         {\jRN{FM}}{1}{1920}{166--209}{#1}}
   \ITEE{#3}{SMazurkiewicz,WSierpinski1920}{
      \BIB{#2}{S. Mazurkiewicz and W. Sierpi\'{n}ski}
         {Contributions a la topologie des ensembles denombrables}
         {\jRN{FM}}{1}{1920}{17--27}{#1}}
   \ITEE{#3}{MMbekhta1992}{
      \BIB{#2}{M. Mbekhta}
         {Sur la structure des composantes connexes semi-Fredholm de $B(H)$}
         {\jRN{PAMS}}{116}{1992}{521--524}{#1}}
   \ITEE{#3}{JEMcCarthy1996}{
      \BIB{#2}{J.E. McCarthy}
         {Boundary values and Cowen-Douglas curvature}
         {\jRN{JFA}}{137}{1996}{1--18}{#1}}
   \ITEE{#3}{JMelleray2007}{
      \BIB{#2}{J. Melleray}
         {Computing the complexity of the relation of isometry between separable Banach spaces}
         {\jRN{MLQ}}{53}{2007}{128--131}{#1}}
   \ITEE{#3}{JMelleray2007a}{
      \BIB{#2}{J. Melleray}
         {On the geometry of Urysohn's universal metric space}
         {\jRN{TopA}}{154}{2007}{384--403}{#1}}
   \ITEE{#3}{JMelleray2008}{
      \BIB{#2}{J. Melleray}
         {Some geometric and dynamical properties of the Urysohn space}
         {\jRN{TopA}}{155}{2008}{1531--1560}{#1}}
   \ITEE{#3}{JMelleray,FVPetrov,AMVershik2008}{
      \BIB{#2}{J. Melleray, F.V. Petrov, A.M. Vershik}
         {Linearly rigid metric spaces and the embedding problem}
         {\jRN{FM}}{199}{2008}{177--194}{#1}}
   \ITEE{#3}{EMichael1953}{
      \BIB{#2}{E. Michael}
         {Some extension theorems for continuous functions}
         {\jRN{PacJM}}{3}{1953}{789--806}{#1}}
   \ITEE{#3}{EMichael1954}{
      \BIB{#2}{E. Michael}
         {Local properties of topological spaces}
         {\jRN{DukeMJ}}{21}{1954}{163--171}{#1}}
   \ITEE{#3}{EMichael1956}{
      \BIB{#2}{E. Michael}
         {Selected selection theorems}
         {\jRN{AmMMon}}{58}{1956}{233--238}{#1}}
   \ITEE{#3}{EMichael1956a}{
      \BIB{#2}{E. Michael}
         {Continuous selections. I}
         {\jRN{AnnM}}{63}{1956}{361--382}{#1}}
   \ITEE{#3}{EMichael1956b}{
      \BIB{#2}{E. Michael}
         {Continuous selections. II}
         {\jRN{AnnM}}{64}{1956}{562--580}{#1}}
   \ITEE{#3}{EMichael1959}{
      \BIB{#2}{E. Michael}
         {A theorem on semi-continuous set-valued functions}
         {\jRN{DukeMJ}}{26}{1959}{647--652}{#1}}
   \ITEE{#3}{JVanMill1986}{
      \BIB{#2}{J. van Mill}
         {Another counterexample in ANR theory}
         {\jRN{PAMS}}{97}{1986}{136--138}{#1}}
   \ITEE{#3}{JVanMill2001}{
      \BIb{#2}{J. van Mill}
         {The Infinite-Dimensional Topology of Function Spaces 
         \textup{(North-Holland Mathematical Library, vol. 64)}}
         {Elsevier, Amsterdam}{2001}{#1}}
   \ITEE{#3}{WMlak1991}{
      \BIb{#2}{W. Mlak}
         {Hilbert Spaces and Operator Theory}
         {PWN --- Polish Scientific Publishers and Kluwer Academic Publishers, Warszawa-Dordrecht}{1991}{#1}}
   \ITEE{#3}{JMogilski1979}{
      \BIB{#2}{J. Mogilski}
         {$CE$-decomposition of $l_2$-manifolds}
         {\jRN{BAPolSSSM}}{27}{1979}{309--314}{#1}}
   \ITEE{#3}{RLMoore1916}{
      \BIB{#2}{R.L. Moore}
         {On the foundations of plane analysis situs}
         {\jRN{TAMS}}{17}{1916}{131--164}{#1}}
   \ITEE{#3}{KMorita1955}{
      \BIB{#2}{K. Morita}
         {A condition for the metrizability of topological spaces and for $n$-dimensionality}
         {\jRN{SciRepTokyoA}}{5}{1955}{33--36}{#1}}
   \ITEE{#3}{AMukherjea,NATserpes1976}{
      \BIb{#2}{A. Mukherjea and N.A. Tserpes}
         {Measures on topological semigroups}
         {Springer Lecture Notes in Math. Vol. 547, Berlin}{1976}{#1}}
   \ITEE{#3}{JMycielski1974}{
      \BIB{#2}{J. Mycielski}
         {Remarks on invariant measures in metric spaces}
         {\jRN{CollM}}{32}{1974}{105--112}{#1}}
   \ITEE{#3}{SNNaboko1984}{
      \BIB{#2}{S.N. Naboko}
         {Conditions for similarity to unitary and selfadjoint operators}
         {\jRN{FunkAnalPril}}{18}{1984}{16--27}{#1}}
   \ITEE{#3}{LNachbin1965}{
      \BIb{#2}{L. Nachbin}
         {The Haar Integral}
         {D. Van Nostrand Company, Inc., Princeton-New Jersey-Toronto-New York-London}{1965}{#1}}
   \ITEE{#3}{TDNarang,SKGarg1991}{
      \BIB{#2}{T.D. Narang and S.K. Garg}
         {On the uniqueness of best approximation in non-archimedian spaces}
         {\jRN{PeriodMHung}}{22}{1991}{121--124}{#1}}
   \ITEE{#3}{JVonNeumann1930}{
      \BIB{#2}{J. von Neumann}
         {Zur Algebra der Funktionaloperationen und Theorie der normalen Operatoren}
         {\jRN{MAnn}}{102}{1930}{370--427}{#1}}
   \ITEE{#3}{JVonNeumann1934}{
      \BIB{#2}{J. von Neumann}
         {Zum Haarschen Mass in topologischen Gruppen}
         {\jRN{ComposM}}{1}{1934}{106--114}{#1}}
   \ITEE{#3}{JVonNeumann1937}{
      \BiB{#2}{J. von Neumann}
         {Some matrix-inequalities and metrization of matrix-space}{\jRN{TomskUnivRev}{} \textbf{1} (1937), 286--300; 
         in }{Collected Works}{Pergamon, New York}{1962}{Vol. 4, 205--219}{#1}}
   \ITEE{#3}{JVonNeumann1949}{
      \BIB{#2}{J. von Neumann}
         {On Rings of Operators. Reduction Theory}
         {\jRN{AnnM}}{50}{1949}{401--485}{#1}}
   \ITEE{#3}{ONielson1973}{
      \BIB{#2}{O. Nielson}
         {Borel sets of von Neumann algebras}
         {\jRN{AmJM}}{95}{1973}{145--164}{#1}}
   \ITEE{#3}{pn2002}{\bibITEM{#2}{#1} \mypaplist{pn1}}
   \ITEE{#3}{pn2006a}{\bibITEM{#2}{#1} \mypaplist{pn2}}
   \ITEE{#3}{pn2006b}{\bibITEM{#2}{#1} \mypaplist{pn3}}
   \ITEE{#3}{pn2007}{\bibITEM{#2}{#1} \mypaplist{pn4}}
   \ITEE{#3}{pn2008a}{\bibITEM{#2}{#1} \mypaplist{pn5}}
   \ITEE{#3}{pn2008b}{\bibITEM{#2}{#1} \mypaplist{pn6}}
   \ITEE{#3}{pn2009a}{\bibITEM{#2}{#1} \mypaplist{pn7}}
   \ITEE{#3}{pn2009b}{\bibITEM{#2}{#1} \mypaplist{pn8}}
   \ITEE{#3}{pn2009c}{\bibITEM{#2}{#1} \mypaplist{pn9}}
   \ITEE{#3}{pn2010a}{\bibITEM{#2}{#1} \mypaplist{pn12}}
   \ITEE{#3}{pn2010b}{\bibITEM{#2}{#1} \mypaplist{pn13}}
   \ITEE{#3}{pn2011a}{\bibITEM{#2}{#1} \mypaplist{pn10}}
   \ITEE{#3}{pn2011b}{\bibITEM{#2}{#1} \mypaplist{pn15}}
   \ITEE{#3}{pn2011c}{\bibITEM{#2}{#1} \mypaplist{pn16}}
   \ITEE{#3}{pn2011d}{\bibITEM{#2}{#1} \mypaplist{pn17}}
   \ITEE{#3}{pn2009x}{
      \bibITEM{#2}{#1} \mypaplist{pn11}}
   \ITEE{#3}{pn2010x}{
      \bibITEM{#2}{#1} \mypaplist{pn14}}
   \ITEE{#3}{pnXXXXb}{
      \bibITEM{#2}{#1} \mypaplist{pnX2}}
   \ITEE{#3}{pnXXXXc}{
      \bibITEM{#2}{#1} \mypaplist{pnX3}}
   \ITEE{#3}{pnXXXXd}{
      \bibITEM{#2}{#1} \mypaplist{pnX13}}
   \ITEE{#3}{MNiezgoda1998}{
      \BIB{#2}{M. Niezgoda}
         {Group majorization and Schur type inequalities}
         {\jRN{LAA}}{268}{1998}{9--30}{#1}}
   \ITEE{#3}{MNiezgoda1998a}{
      \BIB{#2}{M. Niezgoda}
         {An analytical characterization of effective and of irreducible groups inducing cone orderings}
         {\jRN{LAA}}{269}{1998}{105--114}{#1}}
   \ITEE{#3}{MNiezgoda,TYTam2001}{
      \BIB{#2}{M. Niezgoda and T.Y. Tam}
         {On norm property of $G(c)$-radii and Eaton triples}
         {\jRN{LAA}}{336}{2001}{119--130}{#1}}
   \ITEE{#3}{APazy1983}{
      \BIb{#2}{A. Pazy}{Semigroups of Linear Operators 
         and Applications to Partial Differential Equations \textup{(Applied Mathematical Sciences, vol. 44)}}
         {Springer-Verlag, New York}{1983}{#1}}
   \ITEE{#3}{APelc1982}{
      \BIB{#2}{A. Pelc}
         {Semiregular invariant measures on abelian groups}
         {\jRN{PAMS}}{86}{1982}{423--426}{#1}}
   \ITEE{#3}{RPenrose1955}{
      \BIB{#2}{R. Penrose}
         {A generalized inverse for matrices}
         {\jRN{ProcCambPhS}}{51}{1955}{406--413}{#1}}
   \ITEE{#3}{VPestov2006}{
      \BIb{#2}{V. Pestov}
         {Dynamics of infinite-dimensional groups. The Ramsey-Dvoretzky-Milman phenomenon}
         {University Lecture Series \textbf{40}, AMS, Providence, RI}{2006}{#1}}
   \ITEE{#3}{VPestov2007}{
      \BiB{#2}{V. Pestov}
         {Forty-plus annotated questions about large topological groups}
         {in:}{Open Problems in Topology II}{Elliot Pearl (editor), Elsevier B.V., Amsterdam}{2007}{439--450}{#1}}
   \ITEE{#3}{PVPetersen1993}{
      \BiB{#2}{P.V. Petersen}
         {Gromov-Hausdorff convergence of metric spaces}{in book:}{Differential Geometry: Riemannian Geometry 
         (Los Angeles, CA, 1990)}{Amer. Math. Soc., Providence, RI}{1993}{489--504}{#1}}
   \ITEE{#3}{DRamachandran,MMisiurewicz1982}{
      \BIB{#2}{D. Ramachandran and M. Misiurewicz}
         {Hopf's theorem on invariant measures for a group of transformations}
         {\jRN{SM}}{74}{1982}{183--189}{#1}}
   \ITEE{#3}{JMRosenblatt1974}{
      \BIB{#2}{J.M. Rosenblatt}
         {Equivalent invariant measures}
         {\jRN{IsraelJM}}{17}{1974}{261--270}{#1}}
   \ITEE{#3}{HLRoyden1963}{
      \BIb{#2}{H.L. Royden}
         {Real Analysis}
         {The Macmillan Co., New York}{1963}{#1}}
   \ITEE{#3}{WRudin1962}{
      \BIb{#2}{W. Rudin}
         {Fourier Analysis on Groups \textup{(Interscience Tracts in Pure and Applied Mathematics, Number 12)}}
         {Interscience Publishers, New York}{1962}{#1}}
   \ITEE{#3}{WRudin1991}{
      \BIb{#2}{W. Rudin}
         {Functional Analysis}
         {McGraw-Hill Science}{1991}{#1}}
   \ITEE{#3}{TSaito1972}{
      \BiB{#2}{T. Sait\^{o}}{Generations of von Neumann algebras}
         {Lecture Notes in Math. vol. 247}{\textup{(}Lecture on Operator Algebras\textup{)}}
         {Springer, Berlin-Heidelberg-New York}{1972}{435--531}{#1}}
   \ITEE{#3}{KSakai,MYaguchi2003}{
      \BIB{#2}{K. Sakai and M. Yaguchi}
         {Characterizing manifolds modeled on certain dense subspaces of non-separable Hilbert spaces}
         {\jRN{TsukubaJM}}{27}{2003}{143--159}{#1}}
   \ITEE{#3}{SSakai1971}{
      \BIb{#2}{S. Sakai}
         {$\CCc^*$-Algebras and $\WWw^*$-Algebras}
         {Springer-Verlag, Berlin-Heidelberg-New York}{1971}{#1}}
   \ITEE{#3}{RSchori1971}{
      \BIB{#2}{R. Schori}
         {Topological stability for infinite-dimensional manifolds}
         {\jRN{ComposM}}{23}{1971}{87--100}{#1}}
   \ITEE{#3}{JTSchwartz1967}{
      \BIb{#2}{J.T. Schwartz}
         {$\WWw^*$-algebras}
         {Gordon and Breach, Science Publishers Inc., New York-London-Paris}{1967}{#1}}
   \ITEE{#3}{ZSemadeni1971}{
      \BIb{#2}{Z. Semadeni}
         {Banach Spaces of Continuous Functions (Vol. I)}
         {\jRN{PWN}}{1971}{#1}}
   \ITEE{#3}{JPSerre1951}{
      \BIB{#2}{J.-P. Serre}
         {Homologie singuli\`{e}re des espaces fibr\'{e}s}
         {\jRN{AnnM}}{54}{1951}{425--505}{#1}}
   \ITEE{#3}{DSherman2007}{
      \BIB{#2}{D. Sherman}
         {On the dimension theory of von Neumann algebras}
         {\jRN{MScand}}{101}{2007}{123--147}{#1}}
   \ITEE{#3}{WSierpinski1928}{
      \BIB{#2}{W. Sierpi\'{n}ski}
         {Sur les projections des ensembles compl\'{e}mentaires aux ensembles \textup{(A)}}
         {\jRN{FM}}{11}{1928}{117--122}{#1}}
   \ITEE{#3}{MSlocinski1980}{
      \BIB{#2}{M. S\l{}oci\'{n}ski}
         {On the Wold-type decomposition of a pair of commuting isometries}
         {\jRN{APM}}{37}{1980}{255--262}{#1}}
   \ITEE{#3}{RCSteinlage1975}{
      \BIB{#2}{R.C. Steinlage}
         {On Haar measure in locally compact $T_2$ spaces}
         {\jRN{AmJM}}{97}{1975}{291--307}{#1}}
   \ITEE{#3}{JStochel,FHSzafraniec1989}{
      \BIB{#2}{J. Stochel and F.H. Szafraniec}
         {On normal extensions of unbounded operators. III. Spectral properties}
         {\jRN{PublRIMSKyoto}}{25}{1989}{105--139}{#1}}
   \ITEE{#3}{JStochel,FHSzafraniec1989a}{
      \BIB{#2}{J. Stochel and F.H. Szafraniec}
         {The normal part of an unbounded operator}
         {\jRN{ProcKonink}}{92}{1989}{495--503}{#1}}
   \ITEE{#3}{AHStone1962}{
      \BIB{#2}{A.H. Stone}
         {Absolute $\FFf_{\sigma}$-spaces}
         {\jRN{PAMS}}{13}{1962}{495--499}{#1}}
   \ITEE{#3}{AHStone1962a}{
      \BIB{#2}{A.H. Stone}
         {Non-separable Borel sets}
         {\jRN{DissM}}{28}{1962}{41 pages}{#1}}
   \ITEE{#3}{AHStone1972}{
      \BIB{#2}{A.H. Stone}
         {Non-separable Borel sets II}
         {\jRN{GTopA}}{2}{1972}{249--270}{#1}}
   \ITEE{#3}{MHStone1937}{
      \BIB{#2}{M.H. Stone}
         {Application of the theory of Boolean rings to general topology}
         {\jRN{TAMS}}{41}{1937}{375--481}{#1}}
   \ITEE{#3}{MHStone1948}{
      \BIB{#2}{M.H. Stone}
         {The generalized Weierstrass approximation theorem}
         {\jRN{MMag}}{21}{1948}{167--184}{#1}}
   \ITEE{#3}{BSz-Nagy1947}{
      \BIB{#2}{B. Sz.-Nagy}
         {On uniformly bounded linear transformations in Hilbert space}
         {\jRN{ActaSM}}{11}{1947}{152--157}{#1}}
   \ITEE{#3}{WTakahashi1970}{
      \BIB{#2}{W. Takahashi}
         {A convexity in metric space and nonexpansive mappings, I}
         {\jRN{KodaiMSemRep}}{22}{1970}{142--149}{#1}}
   \ITEE{#3}{MTakesaki2002}{
      \BIb{#2}{M. Takesaki}
         {Theory of Operator Algebras I \textup{(Encyclopaedia of Mathematical Sciences, Volume 124)}}
         {Springer-Verlag, Berlin-Heidelberg-New York}{2002}{#1}}
   \ITEE{#3}{MTakesaki2003}{
      \BIb{#2}{M. Takesaki}
         {Theory of Operator Algebras II \textup{(Encyclopaedia of Mathematical Sciences, Volume 125)}}
         {Springer-Verlag, Berlin-Heidelberg-New York}{2003}{#1}}
   \ITEE{#3}{MTakesaki2003a}{
      \BIb{#2}{M. Takesaki}
         {Theory of Operator Algebras III \textup{(Encyclopaedia of Mathematical Sciences, Volume 127)}}
         {Springer-Verlag, Berlin-Heidelberg-New York}{2003}{#1}}
   \ITEE{#3}{TYTam1999}{
      \BIB{#2}{T.Y. Tam}
         {An extension of a result of Lewis}
         {\jRN{ELA}}{5}{1999}{1--10}{#1}}
   \ITEE{#3}{TYTam2000}{
      \BIB{#2}{T.Y. Tam}
         {Group majorization, Eaton triples and numerical range}
         {\jRN{LMLA}}{47}{2000}{11--28}{#1}}
   \ITEE{#3}{TYTam2002}{
      \BIB{#2}{T.Y. Tam}
         {Generalized Schur-concave functions and Eaton triples}
         {\jRN{LMLA}}{50}{2002}{113--120}{#1}}
   \ITEE{#3}{TYTam,WCHill2001}{
      \BIB{#2}{T.Y. Tam and W.C. Hill}
         {On $G$-invariant norms}
         {\jRN{LAA}}{331}{2001}{101--112}{#1}}
   \ITEE{#3}{AFTiman,IAVestfrid1983}{
      \BIB{#2}{A.F. Timan and I.A. Vestfrid}
         {Any separable ultrametric space can be isometrically imbedded in $l_2$}
         {\jRN{FAA}}{17}{1983}{70--71}{#1}}
   \ITEE{#3}{JTomiyama1958}{
      \BIB{#2}{J. Tomiyama}
         {Generalized dimension function for $\WWw^*$-algebras of infinite type}
         {\jRN{TohokuMJ} (2)}{10}{1958}{121--129}{#1}}
   \ITEE{#3}{HTorunczyk1970}{
      \BIB{#2}{H. Toru\'{n}czyk}
         {Remarks on Anderson's paper ``On topological infinite deficiency''}
         {\jRN{FM}}{66}{1970}{393--401}{#1}}
   \ITEE{#3}{HTorunczyk1970a}{
      \BIb{#2}{H. Toru\'{n}czyk}
         {$G$-$K$-absorbing and skeletonized sets in metric spaces}
         {Ph.D. thesis, Inst. Math. Polish Acad. Sci., Warszawa}{1970}{#1}}
   \ITEE{#3}{HTorunczyk1972}{
      \BIB{#2}{H. Toru\'{n}czyk}
         {A short proof of Hausdorff's theorem on extending metrics}
         {\jRN{FM}}{77}{1972}{191--193}{#1}}
   \ITEE{#3}{HTorunczyk1974}{
      \BIB{#2}{H. Toru\'{n}czyk}
         {Absolute retracts as factors of normed linear spaces}
         {\jRN{FM}}{86}{1974}{53--67}{#1}}
   \ITEE{#3}{HTorunczyk1975}{
      \BIB{#2}{H. Toru\'{n}czyk}
         {On Cartesian factors and the topological classification of linear metric spaces}
         {\jRN{FM}}{88}{1975}{71--86}{#1}}
   \ITEE{#3}{HTorunczyk1978}{
      \BIB{#2}{H. Toru\'{n}czyk}
         {Concerning locally homotopy negligible sets and characterization of $l_2$-manifolds}
         {\jRN{FM}}{101}{1978}{93--110}{#1}}
   \ITEE{#3}{HTorunczyk1980}{
      \BiB{#2}{H. Toru\'{n}czyk}{Characterization of infinite-dimensional manifolds}{in:}
         {Proceedings of the International Conference on Geometric Topology (Warsaw, 1978)}
         {\jRN{PWN}}{1980}{431--437}{#1}}
   \ITEE{#3}{HTorunczyk1981}{
      \BIB{#2}{H. Toru\'{n}czyk}
         {Characterizing Hilbert space topology}
         {\jRN{FM}}{111}{1981}{247--262}{#1}}
   \ITEE{#3}{HTorunczyk1985}{
      \BIB{#2}{H. Toru\'{n}czyk}
         {A correction of two papers concerning Hilbert manifolds}
         {\jRN{FM}}{125}{1985}{89--93}{#1}}
   \ITEE{#3}{KTsuda1985}{
      \BIB{#2}{K. Tsuda}
         {A note on closed embeddings of finite dimensional metric spaces}
         {\jRN{BLondMS}}{17}{1985}{273--278}{#1}}
   \ITEE{#3}{PSUrysohn1925}{
      \BIB{#2}{P.S. Urysohn}
         {Sur un espace m\'{e}trique universel}
         {\jRN{CRASParis}}{180}{1925}{803--806}{#1}}
   \ITEE{#3}{PSUrysohn1927}{
      \BIB{#2}{P.S. Urysohn}
         {Sur un espace m\'{e}trique universel}
         {\jRN{BullSM}}{51}{1927}{43--64, 74--96}{#1}}
   \ITEE{#3}{VVUspenskij1986}{
      \BIB{#2}{V.V. Uspenskij}
         {A universal topological group with a countable basis}
         {\jRN{FAA}}{20}{1986}{86--87}{#1}}
   \ITEE{#3}{VVUspenskij1990}{
      \BIB{#2}{V.V. Uspenskij}
         {On the group of isometries of the Urysohn universal metric space}
         {\jRN{CMUC}}{31}{1990}{181--182}{#1}}
   \ITEE{#3}{VVUspenskij2004}{
      \BIB{#2}{V.V. Uspenskij}
         {The Urysohn universal metric space is homeomorphic to a Hilbert space}
         {\jRN{TopA}}{139}{2004}{145--149}{#1}}
   \ITEE{#3}{VVUspenskij2008}{
      \BIB{#2}{V.V. Uspenskij}
         {On subgroups of minimal topological groups}
         {\jRN{TopA}}{155}{2008}{1580--1606}{#1}}
   \ITEE{#3}{VSVaradarajan1963}{
      \BIB{#2}{V.S. Varadarajan}
         {Groups of automorphisms of Borel spaces}
         {\jRN{TAMS}}{109}{1963}{191--220}{#1}}
   \ITEE{#3}{AMVershik1998}{
      \BIB{#2}{A.M. Vershik}
         {The universal Urysohn space, Gromov's metric triples, and random metrics on the series of natural numbers}
         {\jRN{UspekhiMN}}{53}{1998}{57--64}{#1} English translation: \jRN{RussMS}{} \textbf{53} (1998), 921--928. 
         Correction: \jRN{UspekhiMN}{} \textbf{56} (2001), p. 207. English translation: \jRN{RussMS}{} \textbf{56} 
         (2001), p. 1015.}
   \ITEE{#3}{AMVershik2002}{
      \BIb{#2}{A.M. Vershik}
         {Random metric spaces and the universal Urysohn space}
         {Fundamental Mathematics Today. 10th anniversary of the Independent Moscow University. MCCME Publ.}{2002}{#1}}
   \ITEE{#3}{NWeaver1999}{
      \BIb{#2}{N. Weaver}
         {Lipschitz Algebras}
         {World Scientific}{1999}{#1}}
   \ITEE{#3}{JWeidmann1980}{
      \BIb{#2}{J. Weidmann}
         {Linear Operators in Hilbert Spaces}
         {(Graduate Texts in Mathematics, vol. 68) Springer-Verlag New York Inc.}{1980}{#1}}
   \ITEE{#3}{JEWest1969}{
      \BIB{#2}{J.E. West}
         {Approximating homotopies by isotopies in Fr\'{e}chet manifolds}
         {\jRN{BAMS}}{75}{1969}{1254--1257}{#1}}
   \ITEE{#3}{JEWest1969a}{
      \BIB{#2}{J.E. West}
         {Fixed-point sets of transformation groups on infinite-product spaces}
         {\jRN{PAMS}}{21}{1969}{575--582}{#1}}
   \ITEE{#3}{JEWest1970}{
      \BIB{#2}{J.E. West}
         {The ambient homeomorphy of infinite-dimensional Hilbert spaces}
         {\jRN{PacJM}}{34}{1970}{257--267}{#1}}
   \ITEE{#3}{JHCWhitehead1949}{
      \BIB{#2}{J.H.C. Whitehead}
         {Combinatorial homotopy I}
         {\jRN{BAMS}}{55}{1949}{213--245}{#1}}
   \ITEE{#3}{GTWhyburn1942}{
      \BIb{#2}{G. T. Whyburn}
         {Analytic Topology}
         {Amer. Math. Soc. Colloquium Publications (vol. XXVIII), New York}{1942}{#1}}
   \ITEE{#3}{WWogen1969}{
      \BIB{#2}{W. Wogen}
         {On generators for von Neumann algebras}
         {\jRN{BAMS}}{75}{1969}{95--99}{#1}}
   \ITEE{#3}{RYTWong1967}{
      \BIB{#2}{R.Y.T. Wong}
         {On homeomorphisms of certain infinite dimensional spaces}
         {\jRN{TAMS}}{128}{1967}{148--154}{#1}}
   \ITEE{#3}{LYang,JZhang1987}{
      \BIB{#2}{L. Yang and J. Zhang}
         {Average distance constants of some compact convex space}
         {\jRN{JChinUST}}{17}{1987}{17--23}{#1}}
   \ITEE{#3}{PZakrzewski1993}{
      \BIB{#2}{P. Zakrzewski}
         {The existence of invariant $\sigma$-finite measures for a group of transformations}
         {\jRN{IsraelJM}}{83}{1993}{275--287}{#1}}
   \ITEE{#3}{PZakrzewski2002}{
      \BIb{#2}{P. Zakrzewski}
         {Measures on Algebraic-Topological Structures, Handbook of Measure Thoery}
         {E. Pap, ed., Elsevier, Amsterdam}{2002, 1091--1130}{#1}}
   \ITEE{#3}{KZhu2000}{
      \BIB{#2}{K. Zhu}
         {Operators in Cowen-Douglas classes}
         {\jRN{IllinoisJM}}{44}{2000}{767--783}{#1}}
   }

\newcommand{\mypaplist}[2][]{
   \ITEE{#2}{pn1}{
      \myBIB{Separate and joint similarity to families of normal operators}
         {\jRN[#1]{SM}}{149}{2002}{39--62}}
   \ITEE{#2}{pn2}{
      \myBIB{Locally arcwise connected metrizable spaces with the fixed point property are complete-metrizable}
         {\jRN[#1]{TopA}}{153}{2006}{1639--1642}}
   \ITEE{#2}{pn3}{
      \myBIB{Invariant measures for equicontinuous semigroups of continuous transformations of a compact Hausdorff space}
         {\jRN[#1]{TopA}}{153}{2006}{3373--3382}}
   \ITEE{#2}{pn4}{
      \myBIB{Approximation of the Hausdorff distance by the distance of continuous surjections}
         {\jRN[#1]{TopA}}{154}{2007}{655--664}}
   \ITEE{#2}{pn5}{
      \myBIB{Generalized Haar integral}
         {\jRN[#1]{TopA}}{155}{2008}{1323--1328}}
   \ITEE{#2}{pn6}{
      \myBIB{Integration and Lipschitz functions}
         {\jRN[#1]{RCMP}}{57}{2008}{391--399}}
   \ITEE{#2}{pn7}{
      \myBIB{Canonical Banach function spaces generated by Urysohn universal spaces. Measures as Lipschitz maps}
         {\jRN[#1]{SM}}{192}{2009}{97--110}}
   \ITEE{#2}{pn8}{
      \myBIB{Urysohn universal spaces as metric groups of exponent $2$}
         {\jRN[#1]{FM}}{204}{2009}{1--6}}
   \ITEE{#2}{pn9}{
      \myBIB{Central subsets of Urysohn universal spaces}
         {\jRN[#1]{CMUC}}{50}{2009}{445--461}}
   \ITEE{#2}{pn10}{
      \myBIB[P. Niemiec and T.Y. Tam]{A representation of $G$-in\-variant norms for Eaton triple}
         {\jRN[#1]{JCA}}{18}{2011}{59--65}}
   \ITEE{#2}{pn11}{
      \myBIB{Functor of extension of contractions on Urysohn universal spaces}
         {\jRN[#1]{ACS}}{}{2009}{\texttt{DOI: 10.1007/s10485-009-9218-z}}}
   \ITEE{#2}{pn12}{
      \myBIB{Ultra-$\mM$-separability}
         {\jRN[#1]{TopA}}{157}{2010}{669--673}}
   \ITEE{#2}{pn13}{
      \myBIB{Functor of extension of $\Lambda$-isometric maps between central subsets 
         of the unbounded Urysohn universal space}{\jRN[#1]{CMUC}}{51}{2010}{541--549}}
   \ITEE{#2}{pn14}{
      \myBIB{Normed topological pseudovector groups}{\jRN[#1]{ACS}}{}{2010}
         {\ITE{\equal{#1}{}}{\texttt{DOI: 10.1007/s10485\-010-9239-7}}{\texttt{DOI: 10.1007/s10485-010-9239-7}}}}
   \ITEE{#2}{pn15}{
      \myBIB{Topological structure of Urysohn universal spaces}
         {\jRN[#1]{TopA}}{158}{2011}{352--359}}
   \ITEE{#2}{pn16}{
      \myBIB{A note on invariant measures}
         {\jRN[#1]{OpusM}}{31}{2011}{425--431}}
   \ITEE{#2}{pn17}{
      \myBIB{Strengthened Stone-Weierstrass type theorem}
         {\jRN[#1]{OpusM}}{31}{2011}{645--650}}
   \ITEE{#2}{pnX2}{
      \myBAPP{Functor of continuation in Hilbert cube and Hilbert space}
         {to appear in \jRN[#1]{FM}}}
   \ITEE{#2}{pnX3}{
      \myBAPP{Norm closures of orbits of bounded operators}
         {to appear.}}
   \ITEE{#2}{pnX6}{
      \myBAPP{Extending maps by injective $\sigma$-$Z$-maps in Hilbert manifolds}
         {to appear in \jRN[#1]{BullPol}}}
   \ITEE{#2}{pnX7}{
      \myBAPP{Spaces of measurable functions}
         {submitted to \jRN[#1]{CollectM}}}
   \ITEE{#2}{pnX8}{
      \myBAPP{Normal systems over ANR's, rigid embeddings and nonseparable absorbing sets}
         {submitted to \jRN[#1]{ActaMSinES}}}
   \ITEE{#2}{pnX9}{
      \myBAPP{Borel structure of the spectrum of a closed operator}
         {submitted to \jRN[#1]{SM}}}
   \ITEE{#2}{pnX10}{
      \myBAPP{Central points and measures and dense subsets of compact metric spaces}
         {submitted to \jRN[#1]{TopMethNA}}}
   \ITEE{#2}{pnX11}{
      \myBAPP{Generalized absolute values and polar decompositions of a bounded operator}
         {submitted to \jRN[#1]{IEOT}.}}
   \ITEE{#2}{pnX12}{
      \myBAPP{Ultrametrics, extending of Lipschitz maps and nonexpansive selections}
         {accepted for publication in \jRN[#1]{HJM}}}
   \ITEE{#2}{pnX13}{
      \myBAPP{A note on ANR's}
         {submitted to \jRN[#1]{TopA}}}
   \ITEE{#2}{pnX14}{
      \myBAPP{Problem with almost everywhere equality}
         {submitted to \jRN[#1]{ArchM}}}
   \ITEE{#2}{pnX15}{
      \myBAPP{Universal valued Abelian groups}
         {submitted to \jRN[#1]{LNM}}}
   \ITEE{#2}{pnX16}{
      \myBAPP{Unitary equivalence and decompositions of finite systems of closed densely defined operators 
         in Hilbert spaces}{submitted to \jRN[#1]{DissM}}}
   }

\begin{document}

\title{Spaces of measurable functions}
\myData
\begin{abstract}
For a metrizable space $X$ and a finite measure space $(\Omega,\Mm,\mu)$ let $M_{\mu}(X)$ and $M^f_{\mu}(X)$ be
the spaces of all equivalence classes (under the relation of equality almost everywhere mod $\mu$) of $\Mm$-measurable
functions from $\Omega$ to $X$ whose images are separable and finite, respectively, equipped with the topology
of convergence in measure. The main aim of the paper is to prove the following result: if $\mu$ is (nonzero and)
nonatomic and $X$ has more than one point, then the space $M_{\mu}(X)$ is a noncompact absolute retract
and $M^f_{\mu}(A)$ is homotopy dense in $M_{\mu}(X)$ for each dense subset $A$ of $X$. In particular, if $X$ is
completely metrizable, then $M_{\mu}(X)$ is homeomorphic to an infinite-dimensional Hilbert space.
\end{abstract}
\maketitle


In \cite{b-p} Bessaga and Pe\l{}czy\'{n}ski have proved that whenever $X$ is a separable completely metrizable
topological space having more than one point, then the space $M_X$ of Borel functions from $[0,1]$ to $X$
with the topology of convergence in measure is homeomorphic to $l^2$. Later it turned out that the topology
of $l^2$ can be well characterized. This was done by Toru\'{n}czyk\cite{tor1,tor2}. After publication of the latter
papers the number of results on spaces homeomorphic to the separable infinite-dimensional Hilbert space has highly rised.
For example, Dobrowolski and Toru\'{n}czyk\cite{d-t} have shown that every separable completely metrizable non-locally
compact topological group which is an AR is homeomorphic to a Hilbert space. However, the problem whether the assumption
of separability in the latter may be omitted is still open (see \cite{b-z}). In this paper we shall introduce a class
of nonseparable completely metrizable topological groups which are homeomorphic to Hilbert spaces. Namely, if $G$ is
any (nonzero) completely metrizable topological group and $\mu$ is a (nonzero) finite nonatomic measure, then the space
$M_{\mu}(G)$ (defined in Abstract) has a natural structure (induced by the one of $G$) of a topological group and is
homeomorphic to a Hilbert space. In fact we shall prove the following, quite more general, result: if $X$ is a nonempty
metrizable space, $\mu$ is a finite nonatomic measure and $Y = M^r_{\mu}(X)$ is the subspace of $M_{\mu}(X)$ consisting
of all (equivalence classes of) functions whose images are contained in $\sigma$-compact subsets of $X$, then $Y$ is
an absolute retract such that $Y^{\omega} \cong Y$. Since infinite-dimensional Hilbert spaces are the only completely
metrizable noncompact AR's homeomorphic to their own countable infinite Cartesian powers (\cite{tor1}), the latter
mentioned result may be seen as a generalization of earlier results of Bessaga and Pe\l{}czy\'{n}ski\cite{b-p}
as well as of Toru\'{n}czyk\cite{tor3}.\par
Other purpose of the paper is to present the idea of extending maps between metrizable spaces to maps between AR's
via functors. Namely, whenever $\mu$ is a finite (nonzero) nonatomic measure, every map $f\dd X \to Y$ has a natural
extension $M_{\mu}(f)\dd M_{\mu}(X) \to M_{\mu}(Y)$. What is more, the correspondence $f \leftrightarrow M_{\mu}(f)$
preserves many properties (such as: being an injection, an embedding, a map with dense image). We shall show that
if $m$ is the Lebesgue measure on $[0,1]$, the space $W = M_m(X)$ is always an AR satisfying the following conditions:
$X$ is a $Z$-set in $W$ (provided $X$ has more than one point), $W^{\omega} \cong W$, $W$ has RIP and is an S-space
(in the sense of Schori\cite{schori}). As an immediate consequence of this, we shall obtain that if $U$ is a metrizable
manifold modelled on $W$, then $U$ is $W$-stable, i.e. $U \times W \cong U$.\par
Another issue we shall discuss here concerns the question of whether $M_m(M_m(X))$ is homeomorphic to $M_m(X)$.
We shall see that the answer is affirmative for a huge class of metrizable spaces (namely, for spaces in which
every closed separable subset is absolutely measurable), which contains locally absolutely Borel spaces and (separable)
Souslin ones. However, in general we leave this question as an open problem.\par
The article is organized as follows. In the first section we establish notation and terminology, define general spaces
of measurable functions and collect several results on them. Section 2 deals with spaces $M^r_{\mu}(X)$, defined in
this introduction. We show there that if $\mu$ and $\nu$ are two homogeneous (nonatomic) measures of the same weight,
then the spaces $M^r_{\mu}(X)$ and $M^r_{\nu}(X)$ are naturally homeomorphic, whatever $X$ is. The third part is devoted
to spaces of measurable functions over metrizable AM-spaces (i.e. in which every closed separable subset is absolutely
measurable). We prove there that if $X$ is an AM-space, then $M_{\mu}(X) = M^r_{\mu}(X)$ for each finite
measure $\mu$. In Section 4 we state and prove the main result of the paper, which includes the claim that spaces
of measurable functions are absolute retracts. We conclude from this that such spaces over completely metrizable ones
are homeomorphic to Hilbert spaces. In the last part we generalize our results of \cite{pn} to nonseparable case.
Also the idea of extending maps to AR's via the functors $M_{\mu}$ is presented.

\SECT{Preliminaries}

In this paper $\RRR_+$ and $\NNN$ denote the sets of nonnegative reals and integers, respectively, $I = [0,1]$
and $m$ stands for the Lebesgue measure on $I$. If $g$ is any function, $\im g$ stands for the image of $g$.
If, in addition, $g$ takes values in a topological space, $\overline{\im}\, g$ denotes the closure of $\im g$
in the whole space. The weight of a topological space $X$ is denoted by $w(X)$ and is understood as an \textbf{infinite}
cardinal number (i.e. $w(X) = \aleph_0$ for finite $X$). All topological spaces which appear in the paper are metrizable
and all measures are nonnegative, finite and nonzero. For topological spaces $Y$ and $Z$ we shall write $Y \cong Z$
iff $Y$ and $Z$ are homeomorphic. By a \textit{map} we mean a continuous function. If $X$ is a metrizable space,
$X^{\omega}$ stands for the countable infinite Cartesian power of $X$, equipped with the Tichonov topology, and $\Metr(X)$
denotes the family of all bounded metrics on $X$ which induce the given topology of $X$. $\Bb(X)$ stands
for the $\sigma$-algebra of all Borel subsets of $X$, that is, $\Bb(X)$ is the smallest $\sigma$-algebra containing
all open subsets of $X$. If $(\Omega_1 \times \Omega_2,\Mm,\mu)$ is the product space of measure spaces
$(\Omega_1,\Mm_1,\mu_1)$ and $(\Omega_2,\Mm_2,\mu_2)$, then we shall write $\Mm_1 \otimes \Mm_2$ and $\mu_1 \otimes \mu_2$
for $\Mm$ and $\mu$, respectively.\par
Whenever $(\Omega,\Mm)$ is a measurable space and $X$ is a metrizable space, a function $f\dd \Omega \to X$ is
\textit{$\Mm$-measurable}, if $f^{-1}(U) \in \Mm$ for each open subset $U$ of $X$. Sets which are members of $\Mm$
are said to be \textit{measurable}. By a \textit{$\mu$-partition} of $B \in \Mm$ we mean any family $\{B_j\}_{j \in J}$
(with $J \subset \NNN$) of measurable pairwise disjoint sets such that $\mu(B_j) > 0$ for each $j \in J$ and $B =
\bigcup_{j \in J} B_j$. If the images of $\Mm$-measurable functions $f_j\dd \Omega \to X_j$, where $j \in J \subset
\NNN$, are separable, then also the function $\Omega \ni \omega \mapsto (f_j(\omega))_{j \in J} \in \prod_{j \in J} X_j$
is $\Mm$-measurable. Therefore, if $J = \{1,2\}$ and $X_2 = X_1 = X$, the set $\{\omega \in \Omega\dd\ f_1(\omega)
\neq f_2(\omega)\}$ is measurable.\par
We use standard terminology and ideas of measure theory. For details the Reader is referred e.g. to \cite{halmos}.
For example, every measurable function $f\dd \Omega \to X$ with separable image defined on a measure space
$(\Omega,\Mm,\mu)$ will be identified with its equivalence class (in the set of all measurable functions $\Omega \to X$
with separable images) with respect to the relation of almost everywhere equality mod $\mu$. The set of all such
(equivalence classes of) functions is denoted by $M_{\mu}(X)$. The subfamilies of $M_{\mu}(X)$ consisting of all those
functions whose images are, respectively, finite, (at most) countable and contained in $\sigma$-compact subsets of $X$
are denoted by $M^f_{\mu}(X)$, $M^c_{\mu}(X)$ and $M^r_{\mu}(X)$. We clearly have $M^f_{\mu}(X) \subset M^c_{\mu}(X)
\subset M^r_{\mu}(X) \subset M_{\mu}(X)$. Each of the latter inclusions may be proper (the example for the last one
is given in Section~3, see \EXM{notAM}). If $A$ is a subset of $X$, we may and shall naturally identify the members
of $M_{\mu}(A)$ with elements of $M_{\mu}(X)$. Thus, if $N$ stands for $M^f$, $M^c$, $M^r$ or $M$, then $N_{\mu}(A)
\subset N_{\mu}(X)$. Analogously, if $\Nn$ is a $\sigma$-subalgebra of $\Mm$ and $\nu = \mu\bigr|_{\Nn}$,
then for $N=M,M^f,M^c,M^r$ the function $N_{\nu}(X) \ni f \mapsto f \in N_{\mu}(X)$ is well defined (and is isometric
with respect to the metrics $M_{\nu}(d)$ and $M_{\mu}(d)$, defined in the sequel, for every $d \in \Metr(X)$).
The Boolean $\sigma$-algebra (equipped with the metric induced by the measure) associated with a measure space
$(\Omega,\Mm,\mu)$ will be denoted by $\Aa(\mu)$. The weight of $\Aa(\mu)$ is called by us the \textit{weight of $\mu$}
and is denoted by $w(\mu)$. We call the measure $\mu$ \textit{simple} if $\mu(B) \in \{0,\mu(\Omega)\}$
for each $B \in \Mm$ and $\mu$ is \textit{nonatomic} if for every $B \in \Mm$ of positive $\mu$-measure there is a subset
$A \in \Mm$ of $B$ with $0 < \mu(A) < \mu(B)$. Finally, $\mu$ is \textit{homogeneous} if it is nonatomic and for each
$B \in \Mm$ of positive $\mu$-measure, $w(\mu) = w(\mu\bigr|_B)$, where $\mu\bigr|_B = \mu\bigr|_{\Mm_B}$ is a measure
on $B$ and $\Mm_B = \{A \in \Mm\dd\ A \subset B\}$.\par
From now on, we assume that $(\Omega,\Mm,\mu)$ is a measure space with (nonzero) finite measure $\mu$ and that $X$ is
a (nonempty) metrizable space. The space $M_{\mu}(X)$ and all its subsets will always be equipped with the topology
of convergence in measure. In other words, a sequence $(f_n)_n$ of elements of $M_{\mu}(X)$ converges to $f \in M_{\mu}(X)$
iff every its subsequence contains a subsequence $(f_{\nu_n})_n$ such that $f_{\nu_n}(\omega) \to f(\omega)\ (n\to\infty)$
for $\mu$-almost all $\omega \in \Omega$. It is well known that if $\varrho \in \Metr(X)$, then $M_{\mu}(\varrho)
\in \Metr(M_{\mu}(X))$, where $M_{\mu}(\varrho)(f,g) = \int_{\Omega} \varrho(f(\omega),g(\omega)) \dint{\mu(\omega)}$.\par
It is clear that if $(X,\cdot)$ is a metrizable group, then $M_{\mu}(X)$ has a natural topological group structure
(that is, with the pointwise multiplication) induced by the one of $X$.\par
For each $x \in X$ denote by $\delta_{\mu,x} \in M_{\mu}(X)$ the constant function with the only value equal to $x$ and let
$\Delta_{\mu}(X) = \{\delta_{\mu,x}\dd\ x \in X\}$ and $\delta_{\mu,X}\dd X \ni x \mapsto \delta_{\mu,x}
\in \Delta_{\mu}(X) \subset M_{\mu}(X)$.\par
The following are a kind of folklore. Most of them can easily be proved.
\begin{enumerate}[(M1)]
\item $\Delta_{\mu}(X)$ is closed in $M_{\mu}(X)$ and $\delta_{\mu,X}\dd (X,d) \to (\Delta_{\mu}(X),M_{\mu}(d))$
   is an isometry for each $d \in \Metr(X)$. In particular, $\Delta_{\mu}(X) \cong X$. If $X$ is a group,
   $\delta_{\mu,X}$ is a homomorphism.
\item If $\Nn \subset \Mm$ is an algebra of subsets of $X$ which is dense in $\Aa(\mu)$ and $D$ is a dense subset of $X$,
   then the set $M^f(\Nn,D)$ consisting of such functions $f \in M^f_{\mu}(D)$ that $f^{-1}(\{x\}) \in \Nn$ for each
   $x \in D$ is dense in $M_{\mu}(X)$. In particular, $w(M_{\mu}(X)) = \max(w(\mu),w(X))$.
\item If $d \in \Metr(X)$, then $M_{\mu}(d)$ is complete (in the whole space $M_{\mu}(X)$) iff $d$ is complete. The space
   $M_{\mu}(X)$ is completely metrizable iff $X$ is so. Moreover, if $\card X > 1$, then $M_{\mu}(X)$ is noncompact.
\item For each $A \subset X$, $\overline{M_{\mu}(A)} = M_{\mu}(\bar{A})$ (the first closure is in $M_{\mu}(X)$).
\item The measure $\mu$ is nonatomic iff there is a family $\{A_t\}_{t \in I}$ of measurable sets such that
   $A_s \subset A_t$ for $s \leqsl t$ and $\mu(A_t) = t \mu(\Omega)$.
\item If $\mu$ is nonatomic and $\{A_t\}_{t \in I}$ is a family as in (M5), then the map $\lambda\dd
   M_{\mu}(X) \times M_{\mu}(X) \times I \ni (f,g,t) \mapsto f\bigr|_{\Omega \setminus A_t} \cup g\bigr|_{A_t}\in
   M_{\mu}(X)$ is continuous. Moreover, $\lambda(f,g,0) = f$, $\lambda(f,g,1) = g$ and $\lambda(N_{\mu}(X) \times
   N_{\mu}(X) \times I) = N_{\mu}(X)$ for $N = M^f,M^c,M^r$. In particular, each of the spaces $N_{\mu}(X)$
   with $N = M,M^f,M^c,M^r$ is contractible, provided $X$ is nonempty (in fact they are equiconnected).
\item If $\{A_j\}_{j \in J}\ (J \subset \NNN)$ is a $\mu$-partition of $\Omega$; $\lambda = \frac{\mu}{\mu(\Omega)}$
   and $\lambda_j = \frac{\mu\bigr|_{A_j}}{\mu(A_j)}$, then the map $\Phi\dd (M_{\lambda}(X),M_{\lambda}(d)) \ni f
   \mapsto (f\bigr|_{A_j})_{j \in J} \in \bigl(\prod_{j \in J} M_{\lambda_j}(X),\tilde{d}\,\bigr)$ is an isometry,
   where $\tilde{d}((f_j)_{j \in J},(g_j)_{j \in J}) = \sum_{j \in J} \mu(A_j) M_{\lambda_j}(d)(f_j,g_j)$
   for $d \in \Metr(X)$. Moreovoer, $$\Phi(N_{\lambda}(X)) = \prod_{j \in J} N_{\lambda_j}(X)$$ for $N = M^c,M^r$.
   In particular, $N_{\lambda}(X)$ is homeomorphic to $\prod_{j \in J} N_{\lambda_j}(X)$ for $N = M,M^c,M^r$.
\item Let $\{(X_j,d_j)\}_{j \in J}\ (J \subset \NNN)$ be a collection of metric spaces with metrics upper bounded by $1$
   and let $\{a_j\}_{j \in J}$ be a family of positive numbers such that $\sum_{j \in J} a_j < +\infty$.
   Let $X = \prod_{j \in J} X_j$ be a metric space with metric $d((x_j)_{j \in J}, (y_j)_{j \in J}) = \sum_{j \in J}
   a_j d_j(x_j,y_j)$. Analogously, let $D$ be the metric on $\prod_{j \in J} M_{\mu}(X_j)$ given
   by $$D((f_j)_{j \in J},(g_j)_{j \in J}) = \sum_{j \in J} a_j M_{\mu}(d_j)(f_j,g_j).$$ Then the map $$\Psi\dd
   (M_{\mu}(X),M_{\mu}(d)) \ni F \mapsto (p_j \circ F)_{j \in J} \in (\prod_{j \in J} M_{\mu}(X_j),D),$$ where $p_j\dd
   X \to X_j$ is the natural projection, is an isometry. In particular, $M_{\mu}(\prod_{j \in J} X_j)$ is homeomorphic
   to $\prod_{j \in J} M_{\mu}(X_j)$. If $J$ is finite, then
   \begin{equation}\label{eqn:sigma-product}
   \Psi(M^r_{\mu}(X)) = \prod_{j \in J} M^r_{\mu}(X_j)
   \end{equation}
   and $\Psi(M^c_{\mu}(X)) = \prod_{j \in J} M^c_{\mu}(X_j)$.
\item There is a finite or countable collection $\{A_j\}_{j \in J} \cup \{B_k\}_{k \in K}$ (each of $J$ and $K$ may be
   empty) of measurable sets of positive $\mu$-measure such that $\mu\bigr|_{A_j}$ is simple for each $j \in J$,
   while the measures $\mu\bigr|_{B_k}$ with $k \in K$ are homogeneous and of different weights.
\item If $\mu$ is an atom, then $M_{\mu}(X) = M^f_{\mu}(X) = \Delta_{\mu}(X)$ and thus $M_{\mu}(X) \cong X$.
\item (Maharam\cite{maharam}) If $(\Omega_j,\Mm_j,\mu_j)\ (j=1,2)$ are probabilistic spaces such that
   both $\mu_1$ and $\mu_2$ are homogeneous and $w(\mu_1) = w(\mu_2)$, then the Boolean $\sigma$-algebras $\Aa(\mu_1)$
   and $\Aa(\mu_2)$ are isometrically isomorphic.
\end{enumerate}
The property (M1) says that $X$ may naturally be identified (via the map $\delta_{\mu,X}$) with $\Delta_{\mu}(X)$.
The points (M7) and (M9)--(M11) imply that if $N = M, M^c$ or $M^r$, then $N_{\mu}(X) \cong X^p \times \prod_{j \in J}
N_{\mu_j}(X)$, where $p = n \in \NNN$ if $\mu$ has exactly $n$ atoms and $p = \omega$ if $\mu$ has infinitely many atoms
(if $p = 0$, we omit the factor $X^p$); and $J \subset \NNN$ (if $J$ is empty, we omit the factor $\prod_{j \in J}
N_{\mu_j}(X)$) and the measures $\mu_j$ are probabilistic homogeneous and of different weights. We shall prove
in Section~2 that $M^r_{\lambda}(X)$ is \textit{naturally} homeomorphic to $M^r_{\nu}(X)$ if $\lambda$ and $\nu$ are
homogeneous and of the same weight. We shall also show that the connection (\ref{eqn:sigma-product}) is fulfilled
without assumption of finiteness of $J$.\par
Our next aim is to prove that if $(\Omega_j,\Nn_j,\nu_j)$ for $j = 1,2$ are two measure spaces, then there is a measure
space $(\Omega,\Nn,\nu)$ such that $M_{\nu_1}(M_{\nu_2}(X))$ is \textit{naturally} homeomorphic to $M_{\nu}(X)$ for each
metrizable space $X$. To do this, let $\Omega = \Omega_1 \times \Omega_2$ and $\pi\dd \Omega \to \Omega_2$ be the natural
projection. Let $\Nn$ be the $\sigma$-algebra of all subset $A$ of $\Omega$ such that $\pi(A \cap (\{\omega_1\} \times
\Omega_2)) \in \Nn_2$ for each $\omega_1 \in \Omega_1$ and the function $\Omega_1 \ni \omega_1 \mapsto \pi(A \cap
(\{\omega_1\} \times \Omega_2)) \in \Aa(\nu_2)$ is $\Nn_1$-measurable and its image is separable. Finally, let $\nu\dd
\Nn \to \RRR_+$ be given by $\nu(A) = \int_{\Omega_1} \nu_2(\pi(A \cap (\{\omega_1\} \times \Omega_2)))
\dint{\nu_1(\omega_1)}$. It is easy to see that $\Nn$ is indeed a $\sigma$-algebra and that $\nu$ is a finite measure
on $\Omega$. Note also that $\Nn_1 \otimes \Nn_2 \subset \Nn$ and $\nu$ extends $\nu_1 \otimes \nu_2$. We call $\nu$
the \textit{directed product} of $\nu_1$ and $\nu_2$. It would be quite more reasonable to define $(\Omega,\Nn,\nu)$
as the product space of $(\Omega_1,\Nn_1,\nu_1)$ and $(\Omega_2,\Nn_2,\nu_2)$. However, as we will see in Section~3
(\EXM{notAM}), the product space (as $(\Omega,\Nn,\nu)$ below) does not satisfy the following claim:
\begin{enumerate}[(M1)]\setcounter{enumi}{11}
\item For every bounded metric space $(X,d)$ the map $$\Lambda\dd (M_{\nu}(X),M_{\nu}(d)) \to
   (M_{\nu_1}(M_{\nu_2}(X)),M_{\nu_1}(M_{\nu_2}(d)))$$ given by the formula $(\Lambda f(\omega_1))(\omega_2)
   = f(\omega_1,\omega_2)$ is a well defined (bijective) isometry.
\end{enumerate}
To show that $\im \Lambda \subset M_{\nu_1}(M_{\nu_2}(X))$, use the fact that if $f\dd \Omega \to X$ is $\Nn$-measurable
and $\im f$ is separable, then there is a sequence of $\Nn$-measurable functions $f_n\dd \Omega \to X$ with finite
images such that $\lim_{n\to\infty} f_n(\omega) = f(\omega)$ for each $\omega \in \Omega$. Further, direct calculation
shows that $\Lambda$ is isometric. To see the surjectivity, fix an $\Nn_1$-measurable function $g\dd \Omega_1 \to
M_{\nu_2}(X)$ with separable image. Let $\bar{X}$ be the completion of $X$ with respect to $d$. Since $M^f_{\nu_2}(X)$
is dense in $M_{\nu_2}(X)$, there is a sequence of $\Nn_1$-measurable functions $g_n\dd \Omega_1 \to M^f_{\nu_2}(X)$
with finite images such that $\lim_{n\to\infty} g_n(\omega_1) = g(\omega_1)$ for every $\omega_1 \in \Omega_1$.
It is easy to check that for each $n$ there is an $\Nn$-measurable function $f_n\dd \Omega \to X$ whose image is finite
and such that $f_n(\omega_1,\cdot)$ and $g_n(\omega_1)$ concide in $M_{\nu_2}(X)$ for every $\omega_1 \in \Omega_1$
(in fact, each $f_n$ is $\Nn_1 \otimes \Nn_2$-measurable). Thus (since $\Lambda$ is isometric), $(f_n)_n$ is
a fundamental sequence in $M_{\nu}(\bar{X})$. This means that there is an $\Nn$-measurable function $f\dd \Omega \to
\bar{X}$ with separable image which is the limit of $(f_n)_n$ in $M_{\nu}(\bar{X})$. We conclude from this that
$\bar{\Lambda} f = g$, where $\bar{\Lambda}$ is the suitable map `$\Lambda$' for $\bar{X}$. So, the set $A^1 =
\{\omega_1 \in \Omega_1\dd\ f(\omega_1,\cdot) \neq g(\omega_1) \textup{ in } M_{\nu_2}(X)\}$ belongs to $\Nn_1$
and $\nu_1(A^1) = 0$. Now fix $\omega_1 \in \Omega_1 \setminus A^1$. Let $h\dd \Omega_2 \to X$ be an $\Nn_2$-measurable
function with separable image which coincides with $g(\omega_1)$ in $M_{\nu_2}(\Omega_2)$. Then the set $A_{\omega_1}
= \{\omega_2 \in \Omega_2\dd\ f(\omega_1,\omega_2) \neq h(\omega_2)\}$ belongs to $\Nn_2$ and $\nu_2(A_{\omega_1}) = 0$.
Finally, put $A = (A^1 \times \Omega_2) \cup \bigcup_{\omega_1 \in \Omega_1 \setminus A^1} (\{\omega_1\} \times
A_{\omega_1}) \subset \Omega$ and let $f_*\dd \Omega \to X$ be such that $f_*\bigr|_A = f\bigr|_A$
and $f_*\bigr|_{\Omega \setminus A} \equiv b$, where $b$ is a fixed element of $X$. By the construction, $A \in \Nn$,
$f_* \in M_{\nu}(X)$ and $\Lambda(f_*) = g$.\par
The above defined $\sigma$-algebra $\Nn$ and measure $\nu$ will be denoted by us by $\Nn_1 \stackrel{\to}{\otimes} \Nn_2$
and $\nu_1 \stackrel{\to}{\otimes} \nu_2$, respectively. Since $\Aa(\nu_1 \stackrel{\to}{\otimes} \nu_2)$ is naturally
isometric to $M_{\nu_1 \stackrel{\to}{\otimes} \nu_2}(\{0,1\})$, the presented proof of (M12) (especially
$\Nn_1 \otimes \Nn_2$-measurability of the functions $f_n$) yields that
\begin{enumerate}[(M1)]\setcounter{enumi}{12}
\item For each $A \in \Nn_1 \stackrel{\to}{\otimes} \Nn_2$ there is $A_0 \in \Nn_1 \otimes \Nn_2$ such that
   $(\nu_1 \stackrel{\to}{\otimes} \nu_2)(A \setminus A_0) = (\nu_1 \stackrel{\to}{\otimes} \nu_2)(A_0 \setminus A) = 0$.
   In particular, $\Aa(\nu_1 \stackrel{\to}{\otimes} \nu_2) = \Aa(\nu_1 \otimes \nu_2)$ and if $\nu_1$ and $\nu_2$
   are homogeneous, so is $\nu_1 \stackrel{\to}{\otimes} \nu_2$.
\end{enumerate}
Now we shall give a sufficient condition (on a measure $\mu$) under which the space $Y = M_{\mu}(X)$
is homeomorphic to $Y^{\omega}$ (for each $X$). To formulate it, we need an additional notion. We say that two measure
spaces $(\Omega_1,\Mm_1,\mu_1)$ and $(\Omega_2,\Mm_2,\mu_2)$ are \textit{pointwisely isomorphic} if there is a bijection
$\psi\dd \Omega_1 \to \Omega_2$ such that for any $A \subset \Omega_1$, $\psi(A) \in \Mm_2$ iff $A \in \Mm_1$
and $\mu_2(\psi(A)) = \mu_1(A)$ for every $A \in \Mm_1$. In such a situation $\psi$ is called an \textit{isomorphism}.
These spaces are said to be \textit{almost pointwisely isomorphic} if there are sets $A_1 \in \Mm_1$ and $A_2 \in \Mm_2$
such that $\mu_j(\Omega_j \setminus A_j) = 0\ (j=1,2)$ and the spaces $(A_1,\Mm_1\bigr|_{A_1},\mu_1\bigr|_{A_1})$
and $(A_2,\Mm_2\bigr|_{A_2},\mu_2\bigr|_{A_2})$ are \textit{pointwisely isomorphic}. Basicly, every isomorphism
$\varphi\dd \Omega_1 \to \Omega_2$ induces isometries $(M_{\mu_1}(X),M_{\mu_1}(d)) \ni f \mapsto f \circ \varphi^{-1}
\in (M_{\mu_2}(X),M_{\mu_2}(d))$ for any $X$ and $d \in \Metr(X)$ (the same for $M^f$,$M^c$ and $M^r$-spaces).
We also have:
\begin{enumerate}[(M1)]\setcounter{enumi}{13}
\item If there is a measurable set $A$ such that $0 < \mu(A) < \mu(\Omega)$ and the spaces
   $(\Omega,\Mm,\frac{\mu}{\mu(\Omega)})$ and $(A,\Mm\bigr|_A,\frac{\mu|_A}{\mu(A)})$ are almost pointwisely isomorphic,
   then $M_{\mu}(X) \cong M_{\mu}(X)^{\omega}$ for each metrizable space $X$.
\end{enumerate}
To see this, first of all observe that there are measurable sets $\Omega_0$ and $A_0$ such that $A_0 \subset
A \cap \Omega_0$, $\mu(\Omega \setminus \Omega_0) = \mu(A \setminus A_0) = 0$ and the spaces
$(\Omega_0,\Mm\bigr|_{\Omega_0},\mu\bigr|_{\Omega_0})$ and $(A_0,\Mm\bigr|_{A_0},\mu\bigr|_{A_0})$ are pointwisely
isomorphic. (Indeed, if $\tau\dd \Omega_1 \to A_1$ is an isomorphism, where $\Omega_1 \subset \Omega$ and $A_1 \subset A$
are measurable and $\mu(\Omega \setminus \Omega_1) = \mu(A \setminus A_1) = 0$, then for $n \geqsl 2$ put $A_n =
A_{n-1} \cap \Omega_{n-1}$ and $\Omega_n = \tau^{-1}(A_n)$ and finally $A_0 = \bigcap_{n=1}^{\infty} A_n$ and $\Omega_0
= \bigcap_{n=1}^{\infty} \Omega_n$.) Since the maps $(M_{\mu}(X),M_{\mu}(d)) \ni f \mapsto f\bigr|_{\Omega_0} \in
(M_{\mu|_{\Omega_0}}(X),M_{\mu|_{\Omega_0}}(d))$ and $(M_{\mu|_A}(X),M_{\mu|_A}(d)) \ni f \mapsto
f\bigr|_{A_0} \in (M_{\mu|_{A_0}}(X),M_{\mu|_{A_0}}(d))$ are (bijective) isometries for every bounded metric
space $(X,d)$, we may assume that $\Omega_0 = \Omega$ and $A_0 = A$. Let $\varphi\dd \Omega \to A$ be an isomorphism.
For a moment we will think of $\varphi$ as of a function from $\Omega$ to $\Omega$. Let $B_0 = \Omega \setminus A$
and $B_n = \varphi^n(B_0)\ (n \geqsl 1)$, where $\varphi^n$ denotes the $n$-th iterate of $\varphi$. Note that
$\{B_n\}_{n=0}^{\infty}$ is a $\mu$-partition of $B = \bigcup_{n=0}^{\infty} B_n$. What is more, $\varphi(\Omega
\setminus B) = \Omega \setminus B$. But $\frac{\mu(\Omega \setminus B)}{\mu(\Omega)} = \frac{\mu(\varphi(\Omega
\setminus B))}{\mu(A)}$ and thus $\mu(\Omega \setminus B) = 0$. Therefore, as before, we may assume that $B = \Omega$.
Since $\varphi(B_n) = B_{n+1}$, all the spaces $(B_n,\Mm\bigr|_{B_n},\mu\bigr|_{B_n})$ are pointwisely isomorphic.
Take a bijection $\kappa\dd \NNN \times \NNN \to \NNN$ and for each $n,l \in \NNN$ let $\psi_{l,n}\dd B_n \to
B_{\kappa(l,n)}$ be an isomorphism. Finally, for a metrizable space $X$ put $h\dd M_{\mu}(X) \ni f \mapsto
(\bigcup_{n=0}^{\infty} [f\bigr|_{B_{\kappa(l,n)}} \circ \psi_{l,n}])_{l=0}^{\infty} \in M_{\mu}(X)^{\omega}$.
We leave this as a simple exercise that $h$ is a homeomorphism.\par
The point (M14) will be applied in Section~2. We shall end the section with the two more properties of spaces
of measurable functions. Recall that a metrizable space $X$ has the \textit{reflective isotopy property} (in short: RIP)
if there is an ambient invertible isotopy $H\dd X \times X \times I \to X \times X$ such that $H(x,y,0) = (x,y)$
and $H(x,y,1) = (y,x)$ (that is, $H$ needs to be such a homotopy that for each $t \in I$, the map $h_t(x,y) =
H(x,y,t)$ is a homeomorphism of $X \times X$ and the function $(x,y,t) \mapsto h_t^{-1}(x,y)$ is continuous) (compare
\cite[Definition~IX.2.1]{be-pe}). There are other definitions of RIP (see \cite{wong},\cite{west}), all `invertible'
versions of it are however equivalent for spaces $X$ such that $X \cong X^{\omega}$. (M6) implies that:
\begin{enumerate}[(M1)]\setcounter{enumi}{14}
\item If $\mu$ is nonatomic, then the space $N_{\mu}(X)$ has RIP for each metrizable $X$ and $N = M,M^f,M^c,M^r$.
\end{enumerate}
Indeed, if $\{A_t\}_{t \in I}$ is as in (M5), then the map $$H(f,g,t) = (f\bigr|_{\Omega \setminus A_t} \cup
g\bigr|_{A_t},g\bigr|_{\Omega \setminus A_t} \cup f\bigr|_{A_t})$$ is an isotopy we searched for.\par
Following Toru\'{n}czyk\cite{torunczyk}, we say that a closed subset $K$ of a metrizable space $X$ is a \textit{$Z$-set}
if the set $\CCc(Q,X \setminus K)$, where $Q$ is the Hilbert cube, is dense in $\CCc(Q,X)$ in the topology of uniform
convergence. (This definition differs from the original one by Anderson\cite{anderson}, but both these definitions
are equivalent in ANR's.) Countable unions of $Z$-sets are called \textit{$\sigma$-$Z$-sets}. The last property
established in this section, which shall be used in Section~5, is
\begin{enumerate}[(M1)]\setcounter{enumi}{15}
\item Let $\mu$ be nonatomic. If $X$ has more than one point, then $\Delta_{\mu}(X)$ is a $Z$-set in $M_{\mu}(X)$.
If $X$ is infinite, the set $M^f_{\mu}(X)$ is a $\sigma$-$Z$-set in $M_{\mu}(X)$.
\end{enumerate}
We shall prove only the second claim (the first one has similar proof). Let $\{A_t\}_{t \in I}$ be as in (M5).
It is easy to see that for each $n$ the set of all measurable functions whose images have at most $n$ elements is
closed in $M_{\mu}(X)$ and thus $M^f_{\mu}(X)$ is of type $\FFf_{\sigma}$. What is more, there is $u \in M_{\mu}(X)$
such that $u\bigr|_{A_t} \notin M^f_{\mu|_{A_t}}(X)$ for each $t \in I$. Now if $F\dd Q \to M_{\mu}(X)$ is continuous,
then the maps $F_n\dd Q \ni x \mapsto u\bigr|_{A_{1/n}} \cup F(x)\bigr|_{\Omega \setminus A_{1/n}} \in M_{\mu}(X)$
converge uniformly to $F$ and have images disjoint from $M^f_{\mu}(X)$.

\SECT{$M^r$-spaces}

At the beginning we shall study certain spaces of measurable functions.\par
Fix an infinite cardinal number $\alpha$. Each of the sets $I^J$, where $J$ is countable (infinite), will be equipped
with the Tichonov topology. Let $T$ be a set of cardinality $\alpha$. Let $\Omega_{\alpha} = I^T\ (= I^{\alpha})$
and $\Mm_{\alpha}$ be the $\sigma$-algebra of all subsets $B$ of $\Omega_{\alpha}$ for which there are a countable
infinite set $J \subset T$ and $B_0 \in \Bb(I^J)$ such that $B = \{(x_t)_{t \in T}\dd\ (x_j)_{j \in J} \in B_0\}$.
In other words, $\Mm_{\alpha}$ is the product of $\alpha$ copies of $\Bb(I)$. (Note also that, when consider
$\Omega_{\alpha}$ with the Tichonov topology, not every open subset of $\Omega_{\alpha}$ is a member of $\Mm_{\alpha}$.
Open sets which are measurable are exactly those which are $\FFf_{\sigma}$.) Finally, let $m_{\alpha}\dd \Mm_{\alpha}
\to I$ be the product measure of $\alpha$ copies of the Lebesgue measure $m$ on $I$. The following is well known:
\begin{enumerate}[(M1)]\setcounter{enumi}{16}
\item The measure $m_{\alpha}$ is homogeneous and $w(m_{\alpha}) = \alpha$. The measure spaces $(I,\Bb(I),m)$
   and $(\Omega_{\aleph_0},\Mm_{\aleph_0},m_{\aleph_0})$ are pointwisely isomorphic.
\end{enumerate}
We need to know a little bit more about the space $(\Omega_{\alpha},\Mm_{\alpha},m_{\alpha})$. But first a few
necessary definitions.\par
A \textit{Polish} space is a separable completely metrizable one. A subset $B$ of a Polish space $Y$ is said to be
\textit{absolutely measurable in $Y$} if for every probabilistic Borel measure $\mu$ on $Y$ there are two Borel subsets
$A$ and $C$ of $Y$ such that $A \subset B \subset C$ and $\mu(C \setminus A) = 0$. A separable metrizable space $X$
is \textit{absolutely measurable}, if for every embedding $\varphi$ of $X$ into the Hilbert cube $Q$, $\varphi(X)$
is absolutely measurable in $Q$. Equivalently, $X$ is absolutely metrizable if there is $d \in \Metr(X)$ such that $X$
is absolutely measurable in the completion of $(X,d)$.\par
A (separable) \textit{Souslin space} is the empty space or a continuous image of the space of all irrational
numbers; or, equivalently, it is a continuous image of some Polish space. The following are important for us
properties of Souslin spaces:
\begin{enumerate}[(So1)]
\item the image of a Borel function defined on a Borel subset of a Polish space is a Souslin space,
\item every Souslin space is absolutely measurable (compare with \cite[Theorem XIII.4.1]{k-m}).
\end{enumerate}
It is a kind of folklore that every finite Borel measure on a Polish space is regular, i.e. it is supported
on a $\sigma$-compact subset of the whole space. This implies that every finite Borel measure on a (separable)
absolutely measurable space is also supported on a $\sigma$-compact set.\par
All the above facts yield the following result.

\begin{lem}{AMimage}
Let $(\Omega,\Mm,\mu)$ be a finite measure space and let $X$ be a metrizable space.
\begin{enumerate}[\upshape(A)]
\item If the image of an $\Mm$-measurable function $f\dd \Omega \to X$ is contained in a separable absolutely
   measurable subset of $X$, then $f \in M^r_{\mu}(X)$.
\item If $\Nn$ is a $\sigma$-subalgebra of $\Mm$, $\nu = \mu\bigr|_{\Nn}$ and a function $f \in M^r_{\mu}(X)$ belongs
   to the closure of $M_{\nu}(X)$, then $f \in M^r_{\nu}(X)$, i.e. there is an $\Nn$-measurable function $g\dd \Omega
   \to X$ whose image is separable and which is $\mu$-almost everywhere equal to $f$. In particular, $M^r_{\nu}(X)$ is
   closed in $M^r_{\mu}(X)$ and $M^r_{\nu}(X) = M^r_{\bar{\nu}}(X)$, where $\bar{\nu} = \mu\bigr|_{\bar{\Nn}}$
   and $\bar{\Nn}$ consists of those $A \in \Mm$ for which there is $B \in \Nn$ with $\mu(A \setminus B) = \mu(B \setminus
   A) = 0$.
\end{enumerate}
\end{lem}
\begin{proof}
(A): Let $A \subset X$ be a separable absolutely measurable superset of $\im f$. Let $\lambda\dd \Bb(A) \ni B \mapsto
\mu(f^{-1}(B)) \in \RRR_+$. Since $\lambda$ is a finite measure, there is a $\sigma$-compact subset $K$ of $A$ such that
$\lambda(K) = \lambda(A)$. So, $\mu(U) = \mu(\Omega)$ for $U = f^{-1}(K)$. Then $f$ coincides with $\hat{f} \in
M^r_{\mu}(X)$ in $M_{\mu}(X)$, where $\hat{f}\bigr|_U = f\bigr|_U$ and $\hat{f}\bigr|_{\Omega \setminus U} \equiv b$
with $b$ taken from $K$.\par
(B): We only need to prove the first claim. We may assume that the image of $f$ is contained in a $\sigma$-compact
subset of $X$, say $K_0$. By the assumption, there is a sequence of $\Nn$-measurable functions $f_n\dd \Omega \to X$
with finite images which is pointwisely convergent $\mu$-almost everywhere to $f$. Let $K = K_0 \cup \bigcup_n \im f_n$.
Fix $d \in \Metr(K)$ and let $\bar{K}$ be the completion of $(K,d)$. Note that $K$ is $\sigma$-compact and therefore
it is a Borel subset of $\bar{K}$. Let $B$ be the set of all those $\omega \in \Omega$ such that the sequence
$(f_n(\omega))_n$ is convergent in $\bar{K}$. Since $f_n$'s are $\Nn$-measurable, $B \in \Nn$. What is more,
$\mu(\Omega \setminus B) = 0$. Thus, after changing each $f_n$ so that $f_n\bigr|_{\Omega \setminus B} \equiv b$,
where $b \in K$, there is an $\Nn$-measurable function $\bar{g}\dd \Omega \to \bar{K}$ such that $\lim_{n\to\infty}
f_n(\omega) = \bar{g}(\omega)$ for each $\omega$ and $\bar{g}$ is equal to $f$ in $M_{\mu}(\bar{K})$. This yields that
the set $C = \bar{g}^{-1}(K)$ belongs to $\Nn$ and $\mu(\Omega \setminus C) = 0$. Therefore, to end the proof,
it suffices to put $g = \bar{g}\bigr|_C$ and $g\bigr|_{\Omega \setminus C} \equiv b$.
\end{proof}

As we shall see in the next section (\EXM{notAM}), all claims of the point (B) of the above lemma fail when we replace
each $M^r$ by $M$.\par
For a metrizable space $X$ let $M(X) = M_m(X)$ and for an infinite cardinal $\alpha$, let $M_{\alpha}(X) =
M_{m_{\alpha}}(X)$ (analogous notation for metrics). The second claim of (M17) yields that $M_{\aleph_0}(X) \cong M(X)$.

\begin{thm}{M-r}
For every infinite cardinal number $\alpha$ and each metrizable space $X$, $M_{\alpha}(X) = M^r_{m_{\alpha}}(X)$.
\end{thm}
\begin{proof}
We assume that $\Omega_{\alpha} = I^T$. Let $u\dd \Omega_{\alpha} \to X$ be $\Mm_{\alpha}$-measurable with separable
image. Since $u$ is the pointwise limit of a sequence of $\Mm_{\alpha}$-measurable functions with finite images,
we conclude from this that there is a countable infinite set $J \subset T$ such that $u(x) = u(y)$ whenever $x$ and $y$
are elements of $\Omega_{\alpha}$ such that $p_J(x) = p_J(y)$, where $p_J\dd I^T \to I^J$ is the natural projection.
This means that there is a Borel function $v\dd I^J \to X$ such that $u = v \circ p_J$. Let $S = \im v = \im u$.
By (So1) and (So2), $S$ is absolutely measurable and thus \LEM{AMimage}--(A) finishes the proof.
\end{proof}

The argument used in the proof of the above theorem shows also that $M(X) = M^r_m(X)$.\par
Following Schori\cite{schori}, we say that a space $Y$ is an \textit{S-space} if there are an element $\theta \in Y$
and a map $f\dd Y \times I \to Y$ such that:
\begin{enumerate}[(S1)]
\item $f(x,0) = \theta$, $f(x,1) = x$, $f(\theta,t) = \theta$ for each $x \in Y$ and $t \in I$,
\item for every neighbourhood $U$ of $\theta$ in $Y$ there is $t \in (0,1]$ such that $f(Y \times [0,t]) \subset U$,
\item the map $Y \times (0,1] \ni (x,t) \mapsto (f(x,t),t) \in Y \times (0,1]$ is an embedding,
\item $f(f(x,t),s) = f(x,ts)$ for each $t,s \in I$ and $x \in Y$.
\end{enumerate}

\begin{thm}{s-space}
For every infinite cardinal number $\alpha$ and each non\-empty metrizable space $X$, $M_{\alpha}(X)$ is an $S$-space.
\end{thm}
\begin{proof}
As usual, we assume that $\Omega_{\alpha} = I^T$. Fix $\xi \in T$ and $a \in X$ and put $\Omega =
I^{T \setminus \{\xi\}}$, $\theta = \delta_{m_{\alpha},a}$ and $Y = M_{m_{\alpha}}(X)$. We shall identify
$\Omega_{\alpha}$ with $\Omega \times I$. For every $t \in (0,1]$ let $\kappa_t\dd \Omega \times [0,t] \ni (x,s) \mapsto
(x,s/t) \in \Omega \times I$. Finally, let $f\dd Y \times I \ni (u,t) \mapsto (u \circ \kappa_t) \cup
\theta\bigr|_{\Omega \times (t,1]} \in Y$. It is not too difficult to show that $f$ is continuous. What is more,
$f$ satisfies the axioms (S1)--(S4), which finishes the proof.
\end{proof}

It is easily seen that $\mu = m_{\alpha}$ satisfies the assumption of (M14) (for example, look at $\kappa_{1/2}$
defined in the foregoing proof). So, \THM{M-r}, \THM{s-space}, (M14), (M15) and the results of Schori\cite{schori}
imply

\begin{cor}{manifold}
Let $Y = M_{\alpha}(X)$. Then $Y \cong Y^{\omega}$ and for every metrizable manifold $U$ modelled on $Y$, $U \times Y$
is homeomorphic to $U$.
\end{cor}

Before we prove the main result of this section, let us show the following

\begin{pro}{product}
Let $\mu$ be any measure (defined on a $\sigma$-algebra of subsets of $\Omega$) and $X_0,X_1,X_2,\ldots$ be an infinite
sequence of metrizable spaces. Let $J = \NNN$ and $X$ and $\Psi$ be as in \textup{(M8)}. Then $\Psi(M^r_{\mu}(X)) =
\prod_{j \in J} M^r_{\mu}(X_j)$. In other words, for any sequence $(f_n)_{n=0}^{\infty}$ such that $f_n \in M^r_{\mu}(X_n)$
there is $g \in M^r_{\mu}(\prod_{n \in \NNN} X_n)$ such that $(f_n(\omega))_{n=0}^{\infty} = g(\omega)$ for $\mu$-almost
all $\omega \in \Omega$.
\end{pro}
\begin{proof}
Let $f\dd \Omega \ni \omega \mapsto (f_n(\omega))_{n=0}^{\infty} \in \prod_{n \in \NNN} X_n$. Let $\bar{X}_n$ be
the completion of $(X_n,d_n)$, where $d_n$ is a fixed metric on $X_n$. Let $K_n$ be a $\sigma$-compact subset of $X_n$
such that $\im f_n \subset K_n$. Then $K_n \in \Bb(\bar{X}_n)$ and thus $K = \prod_{n=0}^{\infty} K_n \in
\Bb(\prod_{n=0}^{\infty} \bar{X}_n)$. So, $K$ is absolutely measurable and $\im f \subset K$. Now it remains
to apply \LEM{AMimage}--(A).
\end{proof}

And now the main result of the section.

\begin{thm}{isom}
Let $(\Omega_1,\Mm_1,\mu_1)$ and $(\Omega_2,\Mm_2,\mu_2)$ be two nonatomic measure spaces such that $\Aa(\mu_1)$
and $\Aa(\mu_2)$ are isometrically isomorphic. Let $\Phi\dd \Aa(\mu_1) \to \Aa(\mu_2)$ be an isometric isomorphism
of Boolean algebras. Then for every metrizable $X$ there is a unique homeomorphism $H\dd M^r_{\mu_1}(X) \to M^r_{\mu_2}(X)$
such that for each function $f \in M^c_{\mu_1}(X)$ there is a function $g \in M^c_{\mu_2}(X)$ such that $g = H(f)$, $\im g
= \im f$ and $g^{-1}(\{x\}) = \Phi(f^{-1}(\{x\}))$ in $\Aa(\mu_2)$ for every $x \in X$. What is more, $H(\delta_{\mu_1,x})
= H(\delta_{\mu_2,x})$ for each $x \in X$; and for any $d \in \Metr(X)$, $H$ is an isometry with respect to the metrics
$M_{\mu_1}(d)$ and $M_{\mu_2}(d)$.
\end{thm}
\begin{proof}
It is clear that the connections between $f \in M^c_{\mu_1}(X)$ and $g \in M^c_{\mu_2}(X)$ described in the statement
of the theorem well (and uniquely) define $H$ on $M^c_{\mu_1}(X)$. Moreover, in this step $H$ is a bijection between
$M^c$-spaces. It is also clear that $H$ is isometric with respect to the suitable metrics (described in the statement).
Fix $d \in \Metr(X)$ and let $(\bar{X},\bar{d}\,)$ be the completion of $(X,d)$. Since the spaces
$(M_{\mu_j}(\bar{X}),M_{\mu_j}(\bar{d}\,))$ $(j=1,2)$ are complete, there is a unique continuous extension
$$\bar{H}\dd M_{\mu_1}(\bar{X}) \to M_{\mu_2}(\bar{X}),$$ which is simultaneously a (bijective) isometry. It is enough
to check that $\bar{H}(M^r_{\mu_1}(X)) \subset M^r_{\mu_2}(X)$ (because then we infer analogous inclusion
for $\bar{H}^{-1}$). Take an $\Mm_1$-measurable function $f\dd \Omega_1 \to X$ whose image is contained
in a $\sigma$-compact subset of $X$. This implies that there is a $\mu_1$-partition $\{A_n\}_{n=1}^{\infty}$
of $\Omega_1$ such that $K_n = \overline{f(A_n)}$ (the closure taken in $X$) is compact for each $n \geqsl 1$. There is
a $\mu_2$-partition $\{B_n\}_{n=1}^{\infty}$ of $\Omega_2$ such that $B_n = \Phi(A_n)$ in $\Aa(\mu_2)$ for any $n$.
For each $l \geqsl 1$ take a sequence $(f^{(l)}_n\dd A_l \to K_l)_{n=1}^{\infty}$ of $\Mm_1$-measurable functions
with finite images which converges pointwisely to $f\bigr|_{A_l}$. For every $n$ and $l$ let $\im f^{(l)}_n
= \{x^{(l,n)}_1,\ldots,x^{(l,n)}_{p_{l,n}}\}$ and let $B^{(l,n)}_1,\ldots,B^{(l,n)}_{p_{l,n}}$ be a $\mu_2$-partition
of $B_l$ such that $\Phi((f^{(l)}_n)^{-1}(\{x^{(l,n)}_j\})) = B^{(l,n)}_j$ in $\Aa(\mu_2)$. Define $g^{(l)}_n\dd B_l
\to K_l$ in the following way: $g^{(l)}_n\bigr|_{B^{(l,n)}_j} \equiv x^{(l,n)}_j$. Of course $H(\bigcup_{l=1}^{\infty}
f^{(l)}_n) = \bigcup_{l=1}^{\infty} g^{(l)}_n\ (n \geqsl 1)$. So --- since $f_n = \bigcup_{l=1}^{\infty} f^{(l)}_n$
tends to $f$ in $M_{\mu_1}(\bar{X})$ and $\bar{H}$ is isometric --- $g_n = \bigcup_{l=1}^{\infty} g^{(l)}_n$ is
a fundamental sequence in $M_{\mu_2}(\bar{X})$ and thus also the sequence $(g^{(l)}_n)_n$ is fundamental
in $M_{\mu_2|_{B_l}}(\bar{X})$. But $g^{(l)}_n$ is a member of $M_{\mu_2|_{B_l}}(K_l)$, which is closed
in $M_{\mu_2|_{B_l}}(\bar{X})$. This implies that there is $g^{(l)} \in M_{\mu_2|_{B_l}}(K_l)$ which is the limit
of $(g^{(l)}_n)_n$. Then the function $g = \bigcup_{l=1}^{\infty} g^{(l)}$ is the limit of $(g_n)_n$
in $M_{\mu_2}(\bar{X})$. Finally we conclude that $g \in M^r_{\mu_2}(X)$ and $H(f) = g$.
\end{proof}

We shall denote the unique homeomorphism $H$ corresponding to an isometric isomorphism $\Phi$ between Boolean measure
algebras, described in \THM{isom}, by $\widehat{\Phi}$.\par
The above result and (M11) give

\begin{cor}{m-alpha}
If $\mu$ is homogeneous and of weight $\alpha$, then $M^r_{\mu}(X) \cong M_{\alpha}(X)$.
\end{cor}

The note following (M11) combined with the results of this section leads us to

\begin{thm}{nonatomic-r}
Let $\mu$ be nonatomic and let $Y = M^r_{\mu}(X)$. Then there is a finite or countable collection $\{\alpha_j\}_{j \in J}$
of different infinite cardinals such that $Y \cong \prod_{j \in J} M_{\alpha}(X)$. In particular, $Y^{\omega} \cong Y$,
$Y$ has RIP and is an S-space and therefore every metrizable manifold $U$ modelled on $Y$ is $Y$-stable.
\end{thm}

\SECT{AM-class}

\begin{dfn}{AM}
A metrizable space is said to be an \textit{AM-space} [a \textit{So-space}] if every its closed separable subset
is absolutely measurable [a Souslin space].
\end{dfn}

Every So-space is an AM-space and all locally absolutely Borel spaces (in particular, completely metrizable spaces)
are So-spaces. It is also well known that finite or countable Cartesian products of AM-spaces [So-spaces] are AM-spaces
[So-spaces] as well.\par
AM-spaces may be characterized as follows:

\begin{pro}{AM}
For a metrizable space $X$ \tfcae
\begin{enumerate}[\upshape(i)]
\item $X$ is an AM-space,
\item $M^r_{\mu}(X) = M_{\mu}(X)$ for every finite measure space $(\Omega,\Mm,\mu)$,
\item $M^r_{\nu}(X) = M_{\nu}(X)$ for any separable metric space $Y$ and each probabilistic nonatomic measure $\nu$
   defined on $\Bb(Y)$.
\end{enumerate}
\end{pro}
\begin{proof}
Thanks to \LEM{AMimage}--(A), we only need to prove the implication (iii)$\implies$(i). Let $X$ satisfies the claim
of (iii) and let $A$ be a separable closed subset of $X$. Fix $d \in \Metr(A)$ and denote by $\hat{A}$ the completion
of $(A,d)$. Let $\mu$ be a finite Borel measure on $\hat{A}$. We may assume that $\mu$ is nonatomic. Put $\nu\dd \Bb(A)
\ni B \mapsto \inf \{\mu(C)\dd\ C \in \Bb(\hat{A}), B \subset C\} \in \RRR_+$. It is well known that $\nu$ is a measure.
By (iii), there is a Borel function $f\dd A \to X$ whose image is contained in a $\sigma$-compact subset of $X$ and such
that $f(a) = a$ for $\nu$-almost all $a \in A$. Since $A$ is closed in $X$, we may assume that $\im f$ is contained
in a $\sigma$-compact subset of $A$, say $K$. Then $K \in \Bb(\hat{A})$ and $\nu(A \setminus K) = 0$. Clearly, there is
$B \in \Bb(\hat{A})$ such that $A \subset B$ and $\nu(A) = \mu(B)$. Then $K \subset A \subset B$ and $\mu(B \setminus K)
= 0$, which finishes the proof.
\end{proof}

As an application of the above characterization, thanks to (M12), (M13) and the results of Section~2, we obtain

\begin{thm}{M=Mr}
Let $X$ be an AM-space and $d \in \Metr(X)$.
\begin{enumerate}[\upshape(A)]
\item For any finite measure spaces $(\Omega_1,\Mm_1,\mu_1)$ and $(\Omega_2,\Mm_2,\mu_2)$ the map
   $$\Lambda\dd (M_{\mu_1 \otimes \mu_2}(X),M_{\mu_1 \otimes \mu_2}(d)) \to
   (M_{\mu_1}(M_{\mu_2}(X)),M_{\mu_1}(M_{\mu_2}(d)))$$ given by $(\Lambda f(\omega_1))(\omega_2) = f(\omega_1,\omega_2)$
   is a (bijective) isometry.
\item For every two infinite cardinal numbers $\alpha$ and $\beta$, $M_{\alpha}(M_{\beta}(X)) \cong M_{\gamma}(X)$,
   where $\gamma = \max(\alpha,\beta)$. In particular, $M(M(X)) \cong M(X)$.
\item If $(\Omega,\Mm,\mu)$ is a finite measure space, $\Nn$ is a $\sigma$-subalgebra of $\Mm$
   and $\nu = \mu\bigr|_{\Nn}$, then $M_{\nu}(X)$ is closed in $M_{\mu}(X)$.
\item If $\mu$ is a finite nonatomic measure and $Y = M_{\mu}(X)$, then $Y \cong Y^{\omega}$, $Y$ has RIP and is
   an S-space.
\end{enumerate}
\end{thm}

It turns out that the classes of AM-spaces and of So-spaces are invariant under the operators $M_{\mu}$, as it is shown
in the following

\begin{thm}{inv}
If $X$ is an AM-space [a So-space], then $M_{\mu}(X)$ is an AM-space [a So-space] as well for every finite measure
space $(\Omega,\Mm,\mu)$.
\end{thm}
\begin{proof}
Take a separable and closed subset $Y$ of $M_{\mu}(X)$. Let $\{f_n\}_{n=1}^{\infty}$ be a dense subset of $Y$. Put
$A = \overline{\bigcup_{n=1}^{\infty} \im f_n}$ (the closure taken in $X$) and let $\Nn$ be the smallest
$\sigma$-subalgebra of $\Mm$ such that each of the functions $f_n$ is $\Nn$-measurable. Then $A$ is separable
and $\Nn$ is a countably generated $\sigma$-algebra. This means that $\Aa(\nu)$ is separable,
where $\nu = \mu\bigr|_{\Nn}$. Therefore $M_{\nu}(A)$ is separable as well. What is more, by \THM{M=Mr}--(C),
the space $M_{\nu}(A)$ is closed in $M_{\mu}(X)$ and thus $Y \subset M_{\nu}(X)$. Since the classes of AM-spaces
and So-spaces are closed hereditary, it suffices to show that $M_{\nu}(X)$ is an AM-space [a So-space] if so is $X$.
Further, thanks to the note following (M11), we may assume that $\nu$ is nonatomic. But then (see \PRO{AM}
and \COR{m-alpha}) $M_{\nu}(X) \cong M(X)$. So, we have reduced the proof to showing that $M(X)$ is an AM-space
[a So-space], provided $X$ is so and $X$ is separable. First we shall show this for the So-class.\par
Suppose $X$ is a separable nonempty Souslin space. Then there is a continuous surjection $g\dd \RRR \setminus \QQQ
\to X$. Put $M(g)\dd M(\RRR \setminus \QQQ) \ni f \mapsto g \circ f \in M(X)$ (see the last section).
By \cite[Theorem~3.3]{pn}, $M(g)$ is a continuous surjection. So, by the complete metrizability and the separability
of $M(\RRR \setminus \QQQ)$, $M(X)$ is indeed a Souslin space.\par
Now assume that $X$ is a separable absolutely measurable space. Let $S$ be a separable metrizable space and let $\lambda$
be a probabilistic Borel nonatomic measure on $S$. It is enough to prove that $M_{\lambda}(M(X)) = M^r_{\lambda}(M(X))$.
Let $u \in M_{\lambda}(M(X))$. By \THM{M=Mr}--(A), there is a Borel function $v\dd S \times I \to X$ such that
$u(s)$ and $v(s,\cdot)$ coincide in $M(X)$ for $\lambda$-almost all $s \in S$. Since $X$ is absolutely measurable,
there is a Borel function $w\dd S \times I \to X$ whose image is contained in a $\sigma$-compact subset of $X$
(say $K$) and such that $v$ and $w$ coincide in $M_{\lambda \otimes m}(X)$. Put $\widetilde{u}\dd S \ni s \mapsto
w(s,\cdot) \in M(K) \subset M(X)$. Then $\widetilde{u}$ and $u$ represent the same element of $M_{\lambda}(M(X))$.
What is more, $\widetilde{u} \in M_{\lambda}(M(K))$ and $M(K)$ is a Souslin space, which yields that $\widetilde{u}
\in M^r_{\lambda}(M(K)) \subset M^r_{\lambda}(M(X))$. This finishes the proof.
\end{proof}

We end the section with the following

\begin{exm}{notAM}
It is well known that there exists a subset $X$ of the square $I^2$ which is not Lebesgue measurable, but for
each $t \in I$ the set $X_t = \{s \in I\dd\ (t,s) \in X\}$ is a Borel subset of $I$ and $m(X_t) = 1$.
This implies that $X \in \Bb(I) \stackrel{\to}{\otimes} \Bb(I)$. So, the map $f\dd I^2 \to X$
which is the identity on $X$ and constant on its complement is $\Bb(I) \stackrel{\to}{\otimes} \Bb(I)$-measurable.
However, since $X$ is nonmeasurable, there is no $g \in M_{m \otimes m}(X)$ which coincides with $f$
in $M_{m \stackrel{\to}{\otimes} m}(X)$; and $f \notin M^r_{m \stackrel{\to}{\otimes} m}(X)$.
Thus we have obtained that $M^r_{m \stackrel{\to}{\otimes} m}(X) \varsubsetneq M_{m \stackrel{\to}{\otimes} m}(X)$
and $M_{m \otimes m}(X) \varsubsetneq M_{m \stackrel{\to}{\otimes} m}(X)$ as well. The example shows that (M12)
is in general not true if we put there $\nu = \nu_1 \otimes \nu_2$. It also shows that if $\nu$ is the restriction
of $\mu$ to a dense (in $\Aa(\mu)$) $\sigma$-subalgebra, then $M_{\nu}(X)$, in spite of its density in $M_{\mu}(X)$,
may differ from $M_{\mu}(X)$.
\end{exm}

We do not know whether $M(M(X)) \cong M(X)$ if $X$ is as in \EXM{notAM}.

\SECT{Main results}

In this section we shall show that all considered by us spaces of measurable functions with respect to nonatomic
measures are absolute retracts. In our proof we shall use the following three results:

\begin{lemm}{b-p1}{\mbox{\cite[Theorem~3.1]{b-p}}}
Every metrizable space admits an embedding into the unit sphere of a Hilbert space whose image is linearly independent.
\end{lemm}

\begin{lemm}{b-p2}{\mbox{\cite[the proof of Lemma~4.3]{b-p}}}
Let $T$ be a finite linearly independent subset of the unit sphere of a Hilbert space $(\HHh,\scalarr)$. Let $K$
be the convex hull of $T$ (in $\HHh$) and let $D = M(K)$ be equipped with the topology $\tau_w$ induced by the weak one
of $L^2[\lin T]\ (= L^2(m,\lin T))$, i.e. a sequence $(f_n)_n$ of members of $D$ converges to $f \in D$ iff $$\int_0^1
\scalar{f_n(t)}{g(t)} \dint{t} \to \int_0^1 \scalar{f(t)}{g(t)} \dint{t}\ (n \to \infty)$$ for each $g \in D$. Then:
\begin{enumerate}[\upshape(BP1)]
\item $(D,\tau_w)$ is metrizable compact and convex,
\item the inclusion map of $M(T)$ into $D$ is an embedding, when $M(T)$ is equipped with the topology of convergence
   in measure,
\item there is a sequence of maps from $D$ into $D$ whose images are contained in $M(T)$ which is uniformly convergent
   to the identity map on $D$.
\end{enumerate}
\end{lemm}

\begin{thmm}{tor}{\mbox{\cite[Theorem~1.1]{tor0}}}
If a metrizable space $X$ has a basis (consisting of open sets) such that all finite intersections of its members
are homotopically trivial, then $X$ is an ANR.
\end{thmm}

Recall that a topological space $X$ is \textit{homotopically trivial} iff for every $n \geqsl 1$ each map of $\partial
I^n$ into $X$ is extendable to a map of $I^n$ into $X$. Note also that the empty space is homotopically trivial.\par
Following \cite{brz},\cite{banakh}, a subset $A$ of a space $X$ is said to be \textit{homotopy dense} (in $X$)
if there is a homotopy $H\dd X \times [0,1] \to X$ such that $H(x,0) = x$ for each $x \in X$ and $H(X \times (0,1])
\subset A$. If $X$ is an ANR, then $A$ is homotopy dense in $X$ iff $X \setminus A$ is locally homotopy negligible
in $X$ (\cite{tor0}). The main result of the paper has the following form:

\begin{thm}{main}
Let $(\Omega,\Mm,\mu)$ be a finite nonatomic measure space, $X$ a nonempty metrizable space and $A$ its dense subset.
Then the space $M_{\mu}(X)$ is an AR and $M^f_{\mu}(A)$ is homotopy dense in $M_{\mu}(X)$.
\end{thm}

The proof of the above theorem is divided into a few lemmas. Let us fix a finite nonatomic measure space
$(\Omega,\Mm,\mu)$, a nonempty metrizable space $X$ and its dense subset $A$. By \LEM{b-p1}, we may assume that $X$
is a linearly independent subset of the unit sphere $\SSs$ of a Hilbert space $(\HHh,\scalarr)$. For each bounded
subset $E$ of $\HHh$, the topology of convergence in measure in $M_{\mu}(E)$ coincides with the topology induced
by the metric $\varrho_E(u,v) = (\int_{\Omega} \|u(\omega) - v(\omega)\|^2 \dint{\mu(\omega)})^{1/2}\
(u,v \in M_{\mu}(E))$. For each $Y \subset X$, we shall denote by $B_{\varrho_Y}(u,r)$ the open ball
in $(M_{\mu}(Y),\varrho_Y)$ with center at $u \in M_{\mu}(Y)$ and of radius $r > 0$. Our purpose is to prove that
if $u_1,\ldots,u_p \in M_{\mu}(X)$ and $r_1,\ldots,r_p > 0$, then $\bigcap_{j=1}^p B_{\varrho_X}(u_j,r_j)$
is homotopically trivial. First we shall show a special case of this.

\begin{lem}{finite}
Let $T$ be a finite subset of $X$. Then for every $u_1,\ldots,u_p \in M_{\mu}(T)$ and each $r_1,\ldots,r_p > 0$, the set
$G = \bigcap_{j=1}^p B_{\varrho_T}(u_j,r_j)$ is homotopically trivial.
\end{lem}
\begin{proof}
Fix $k \geqsl 1$ and take a map $f\dd \partial(I^k) \to G$. Let $\{A_t\}_{t \in I}$ be as in (M5) and let $E$
be an at most countable dense subset of $\im f$. There is a countably generated $\sigma$-subalgebra $\Nn$ of $\Mm$
such that each member of $E$ is $\Nn$-measurable and $A_q \in \Nn$ for $q \in \QQQ \cap I$. Put $\nu = \mu\bigr|_{\Nn}$.
Then $\nu$ is nonatomic and $\Aa(\nu)$ is separable. What is more, since $T$ is clearly an AM-space, $M_{\nu}(T)$ is
closed in $M_{\mu}(T)$ (\THM{M=Mr}--(C)). This implies that $\im f \subset M_{\nu}(T)$. Now by \COR{m-alpha},
$M_{\nu}(T) \cong M(T)$. Note also that the homeomorphism $H$ (between $M_{\nu}(T)$ and $M(T)$) appearing in the statement
of \THM{isom} is an isometry with respect to the metrics $\varrho_T\bigr|_{M_{\nu}(T) \times M_{\nu}(T)}$ and
$d_T\dd M(T) \times M(T) \ni (u,v) \mapsto (\int_0^1 \|u(t) - v(t)\|^2 \dint{t})^{1/2} \in \RRR_+$. So, the inverse
image of $G$ under $H$ coincides with the finite intersection of open $d_T$-balls in $M(T)$. This reduces the problem
to the case when $(\Omega,\Mm,\mu) = (I,\Bb(I),m)$, which we now assume. Let $V = \lin T \subset \HHh$. Following
Bessaga and Pe\l{}czy\'{n}ski\cite{b-p}, consider the Hilbert space $L^2[V]$ of all (equivalence classes of) Borel
functions $w\dd I \to V$ such that $\int_0^1 \|w(t)\|^2 \dint{t} < +\infty$ with the scalar product $\scalar{u}{v}_V
= \int_0^1 \scalar{u(t)}{v(t)} \dint{t}$. Put $\delta_j = 1 - \frac{r_j^2}{2}$ and $U_j = \{g \in L^2[V]\dd\
\scalar{g}{u_j}_V > \delta_j\}\ (j=1,\ldots,p)$. It is easily seen that each $U_j$ is convex and open in the weak
topology of $L^2[V]$. Let $K$ and $(D,\tau_w)$ be as in the statement of \LEM{b-p2}. By (BP2), the topology of $M(T)$
coincides with the one induced by $\tau_w$ and therefore to the end of the proof we shall deal only with the topology
$\tau_w$. Put $U = D \cap \bigcap_{j=1}^p U_j$. Note that $U$ is open in $(D,\tau_w)$ and $U$ is convex. What is more,
since $T \subset \SSs$, $M(T)$ is contained in the unit sphere of $L^2[V]$ and therefore
\begin{equation}\label{frag}
U \cap M(T) = G.
\end{equation}
This implies that $f\dd \partial(I^k) \to U$. Since $U$ is convex, there exists a continuous extension $\hat{f}\dd I^k
\to U$ of $f$. Further, applying \LEM{b-p2}, take a sequence of maps $\varphi_n\dd D \to D$ which is uniformly convergent
to the identity map on $D$ and such that $\im \varphi_n \subset M(T)$. Then the sequence $f_n = \varphi_n \circ
\hat{f}\dd I^k \to D$ is uniformly convergent to $\hat{f}$. This yields that for infinitely many $n$ we have
$\im f_n \subset U$ and thus we may assume that the latter inclusion is satisfied for each $n$. But $\im f_n \subset
\im \varphi_n \subset M(T)$, which combined with (\ref{frag}) gives $\im f_n \subset G$. Again by (BP2), the sequence
$(f_n\bigr|_{\partial(I^k)}\dd \partial(I^k) \to G)_n$ tends uniformly to $f$ (with respect to the topology of $M(T)$).
Finally, since $G$ is open in $M(T)$ and thanks to (M6), $G$ is locally equiconnected, which implies that for some $n$,
$f_n\bigr|_{\partial(I^k)}$ and $f$ are homotopic in $G$. So, the homotopy extension property finishes the proof.
\end{proof}

The main result (\THM{main}) is an easy consequence of \THM{tor} and the following

\begin{lem}{main}
If $u_1,\ldots,u_p \in M_{\mu}(X)$ and $r_1,\ldots,r_p > 0$, then the set $W = \bigcap_{j=1}^p B_{\varrho_X}(u_j,r_j)$
is homotopically trivial.
\end{lem}
\begin{proof}
We may assume that $\mu$ is probabilistic. For each $k \geqsl 1$ let $\Delta_k = \{(t_0,\ldots,t_k) \in I^{k+1}\dd\
\sum_{j=0}^k t_j = 1\}$ be the $k$-dimensional simplex and let $\partial(\Delta_k) = \{x \in \Delta_k\dd\ x_j = 0
\textup{ for some } j\}$ be its combinatorial boundary. It is enough to prove that each map of $\partial(\Delta_k)$
into $W$ is extendable to a map of $\Delta_k$ into $W$. Fix a map $f\dd \partial(\Delta_k) \to W$. Since $W$ is open,
$\im f$ is compact and $M^f_{\mu}(A)$ is dense in $M_{\mu}(X)$, there are functions $u_1^*,\ldots,u_p^* \in M^f_{\mu}(A)$
and numbers $r_1^*,\ldots,r_p^* > 0$ such that $\im f \subset W^* \subset W$, where $W^* = \bigcap_{j=1}^p
B_{\varrho_X}(u_j^*,r_j^*)$. Thus we may and shall assume that
\begin{equation}\label{eqn:f}
u_1,\ldots,u_p \in M^f_{\mu}(A).
\end{equation}
As in the proof of \LEM{finite}, take a family $\{A_t\}_{t \in I}$ satisfying the claim of (M5). One may show that
for each $n \geqsl 1$ the function $\lambda_n\dd M^f_{\mu}(A)^{n+1} \times \Delta_n \to M^f_{\mu}(A)$ given by
$$\lambda_n(v_0,\ldots,v_n;t_0,\ldots,t_n) = \bigcup_{j=0}^n v_j\bigr|_{A_{t_j} \setminus A_{t_{j-1}}}$$
(with $t_{-1} = 0$) is continuous. Further, take a positive number $\epsi$ such that
\begin{equation}\label{eqn:epsi}
\bigcup_{x \in \partial(\Delta_k)} B_{\varrho_X}(f(x),\epsi) \subset W
\end{equation}
and put $\delta = \frac{\epsi}{3 \sqrt{k}}$. Fix $l \geqsl 1$. Let $\KKk_0$ be the collection of all faces
of $\Delta_k$ and for each $n \geqsl 1$ let $\KKk_n$ be the collection of all geometric simplices obtained
by the barycentric divisions of all members of $\KKk_{n-1}$. There is $N \geqsl 1$ such that $\diam_{\varrho_X}f(\sigma)
\leqsl \frac{\delta}{l}$ for every $\sigma \in \KKk_N$. Now for each vertex $x$ of any member of $\KKk = \KKk_N$ take
$v_x \in M^f_{\mu}(A)$ such that $\varrho_X(f(x),v_x) \leqsl \delta/l$. Let `$\preccurlyeq$' be a total order on the set
of all vertices of all members of $\KKk$. Take any $\sigma \in \KKk$ and assume that $x_0 \prec \ldots \prec x_k$
are vertices of $\sigma$. We define $g_{\sigma}\dd \sigma \to M^f_{\mu}(A)$ by $g_{\sigma}(\sum_{j=0}^k t_j x_j)
= \lambda_k(v_{x_0},\ldots,v_{x_k};t_0,\ldots,t_k)\ ((t_0,\ldots,t_k) \in \Delta_k)$. Since $x_0,\ldots,x_k$
are linearly independent, $g_{\sigma}$ is continuous. What is more, if also $\sigma' \in \KKk$, then $g_{\sigma}$
and $g_{\sigma'}$ coincide on $\sigma \cap \sigma'$. This yields that the union $g_l$ of all $g_{\sigma}$'s is a well
defined continuous function from $\partial(\Delta_k)$ into $M^f_{\mu}(A)$. And, what is important, there is a finite
subset $T_l$ of $A$ such that $\im g_l \subset M_{\mu}(T_l)$. Moreover, if $x \in \sigma$, where $\sigma \in \KKk$
has vertices $x_0 \prec \ldots \prec x_k$, then
\begin{multline}\label{eqn:estim}
\varrho_X(g_l(x),f(x))^2 \leqsl \sum_{j=0}^k \varrho_X(v_{x_j},f(x))^2\\ \leqsl 2\sum_{j=0}^k
\Bigl(\varrho_X(v_{x_j},f(x_j))^2 + \varrho_X(f(x_j),f(x))^2\Bigr) \leqsl 4k \frac{\delta^2}{l^2} < \frac{\epsi^2}{l^2}
\end{multline}
and thus, by (\ref{eqn:epsi}), $\im g_l \subset W$.\par
Now for $t \in (l,l+1)$ put $g_t\dd \partial(\Delta_k) \ni x \mapsto \lambda_1(g_l(x),g_{l+1}(x);l+1-t,t-l)
\in M^f_{\mu}(A)$. It is clear that $(g_t)_{t \geqsl 1}$ is a homotopy. Furthermore, one checks, using (\ref{eqn:estim}),
that for each $t \in [l,l+1]$ with $l \geqsl 2$ and every $x \in \partial(\Delta_k)$, $\varrho_X(g_t(x),f(x))^2
< \frac{2\epsi^2}{l^2} \leqsl \epsi^2$. So, $\im g_t \subset W$ for $t \geqsl 2$ and the function $h\dd \partial(\Delta_k)
\times [2,\infty] \to W$ given by $h(x,t) = g_t(x)$ for $t < +\infty$ and $h(x,\infty) = f(x)$ is a homotopy connecting
$g_2$ and $f$. By the homotopy extension property, it suffices to show that $g_2$ is extendable to a map of $\Delta_k$
into $W$. To this end, put $T = T_2 \cup \bigcup_{j=1}^p \im u_j \subset A$. By (\ref{eqn:f}), $T$ is finite. What is
more, $u_1,\ldots,u_p \in M_{\mu}(T)$. Finally, $W \cap M_{\mu}(T) = \bigcap_{j=1}^p B_{\varrho_T}(u_j,r_j)$ and $g_2\dd
\partial(\Delta_k) \to W \cap M_{\mu}(T)$. So, by \LEM{finite}, $g_2$ admits a continuous extension of $\Delta_k$
into $W \cap M_{\mu}(T)$, which finishes the proof.
\end{proof}

\begin{proof}[Proof of \THM{main}]
Let $X$ be embedded as a linearly independent subset of the unit sphere of a Hilbert space. By \THM{tor} and \LEM{main},
$M_{\mu}(X)$ is a homotopically trivial ANR. This yields that it is an AR. Finally, the last paragraph of the proof
of \LEM{main} shows that for every open ball $B$ in $M_{\mu}(X)$ (with respect to the metric $\varrho_X$) the inclusion
map $B \cap M^f_{\mu}(A) \to B$ is a (weak) homotopy equivalence and hence $M_{\mu}(X) \setminus M^f_{\mu}(A)$ is
locally homotopy negligible in $M_{\mu}(X)$ (\cite{tor0}).
\end{proof}

\begin{cor}{f,c,r}
If $\mu$ is a finite nonatomic measure and $X$ is a nonempty metrizable space, then the spaces $M^f_{\mu}(X)$,
$M^c_{\mu}(X)$, $M^r_{\mu}(X)$ and $M_{\mu}(X)$ are AR's.
\end{cor}

\begin{rem}{algebra}
If $\Nn$, $A$ and $M^f(\Nn,A)$ are as in (M2) and additionally $\Nn$ contains a subfamily $\{A_t\}_{t \in I}$ as in (M5),
the proof of \LEM{main} shows that $M^f(\Nn,A)$ is homotopy dense in $M_{\mu}(X)$ and thus it is an AR. This implies
that the space $M^s(X) \subset M(X)$ of all piecewise constant functions is an AR.
\end{rem}

As a first consequence of \THM{main} we obtain a generalization of theorems of Bessaga and Pe\l{}czy\'{n}ski\cite{b-p}
and of Toru\'{n}czyk\cite{tor3}:

\begin{thm}{H}
If $\mu$ is a finite nonatomic (nonzero) measure and $X$ is a completely metrizable space which has more than one point,
then $M_{\mu}(X)$ is homeomorphic to an infinite-dimensional Hilbert space of dimension $\alpha = \max(w(\mu),w(X))$.
\end{thm}
\begin{proof}
Put $Y = M_{\mu}(X)$. By \THM{M=Mr}--(D), $Y^{\omega} \cong Y$. But $Y$ is a noncompact AR and thus,
by \cite[Theorem~5.1]{tor1}, $Y$ is homeomorphic to a Hilbert space of dimension $w(Y)$. So, the observation that
$w(Y) = \alpha$ finishes the proof.
\end{proof}

Now repeating the proofs (with $M_G$ replaced by $M_{\alpha}(G)$) of Theorem 5.1 and Corollary 5.2 of \cite{b-p} we get

\begin{cor}{action}
Let $\HHh$ be a Hilbert space of dimension $\alpha \geqsl \aleph_0$ and let $G$ be a completely metrizable topological
group of weight no greater than $\alpha$. Then $G$ is (algebraically and topologically) isomorphic to a closed subgroup
of a group homeomorphic to $\HHh$ and $G$ admits a free action in $\HHh$.
\end{cor}

\SECT{Extending maps}

We begin this section with

\begin{dfn}{functor}
Let $\mu$ be a finite measure and let $f\dd X \to Y$ be a map. Let
$$
M_{\mu}(f)\dd M_{\mu}(X) \ni u \mapsto f \circ u \in M_{\mu}(Y).
$$
$M_{\mu}(f)$ is said to be the \textit{$\mu$-extension of $f$}. Additionally, let $M(f) = M_m(f)$
and $M_{\alpha}(f) = M_{m_{\alpha}}(f)$ for every infinite cardinal $\alpha$.
\end{dfn}

Note that $M_{\mu}(f)$ is continuous and that $M_{\mu}(f)(N_{\mu}(X)) \subset N_{\mu}(Y)$ for $N=M^f,M^c,M^r$.
The connection
\begin{equation}\label{eqn:ext}
M_{\mu}(f) (\delta_{\mu,X}(x)) = \delta_{\mu,Y}(f(x)) \qquad (x \in X)
\end{equation}
says that $M_{\mu}(f)$ extends $f$, when we identify the elements of $Z$ with the ones of $\Delta_{\mu,Z}$ via
$\delta_{\mu,Z}$ with $Z=X,Y$, which justifies the undertaken terminology. If, in addition, $X$ and $Y$ are
topological groups and $f$ is a group homomorphism, so is $M_{\mu}(f)$.\par
The Reader will easily check that whenever $\mu$ is a fixed finite measure, the operations $X \mapsto M_{\mu}(X)$
and $f \mapsto M_{\mu}(f)$ define a functor. This functor has interesting properties, whose proofs are left
as exercises (below we assume that $g_n,g\dd X \to Y$ are maps):
\begin{enumerate}[(F1)]
\item $M_{\mu}(g)$ is an injection [embedding] iff $g$ is so,
\item $\overline{\im}\,M_{\mu}(g) = M_{\mu}(\,\overline{\im}\,g)$,
\item the sequence $(M_{\mu}(g_n))_n$ is pointwisely [uniformly on compact subsets of $M_{\mu}(X)$] convergent
   to $M_{\mu}(g)$ iff the sequence $(g_n)_n$ pointwisely [uniformly on compact subsets of $X$] converges to $g$,
\item for each $\varrho \in \Metr(Y)$ the map $$(\CCc(X,Y),\varrho_{\sup}) \ni h \mapsto M_{\mu}(h)
   \in (\CCc(M_{\mu}(X),M_{\mu}(Y)),(M_{\mu}(\varrho))_{\sup})$$ is isometric (`$\CCc(A,B)$' denotes the collection
   of all maps from $A$ to $B$ and `$d_{\sup}$' stands for the supremum metric induced by a bounded metric $d$).
\end{enumerate}
It is clear that for each $f \in \CCc(X,Y)$, $\im M_{\mu}(f) \subset \bigcup_A M_{\mu}(f(A))$ where $A$ runs over all
separable closed subsets of $X$. We do not know whether the latter inclusion can always be replaced by the equality.
We are only able to show the following result, the proof of which is similar to the proof of \cite[Theorem~3.3]{pn}.

\begin{pro}{image}
Whenever $\mu$ is a finite measure and $f\dd X \to Y$ is a map, $\im M^r_{\mu}(f) = \bigcup_K M^r_{\mu}(f(K))$
where $K$ runs over all $\sigma$-compact subsets of $X$ and $M^r_{\mu}(f) = M_{\mu}(f)\bigr|_{M^r_{\mu}(X)}$.
What is more, if $v \in M^r_{\mu}(f(A))$, where $A$ is a (separable) Souslin subset of $X$,
then $v \in \im M^r_{\mu}(f)$.
\end{pro}
\begin{proof}
We only need to prove the second claim. Put $C = f(A)$ and let $L$ be a $\sigma$-compact subset of $C$ such that
$\im v \subset L$. Let $\nu\dd \Bb(L) \ni B \mapsto \mu(v^{-1}(B)) \in \RRR_+$. Put $K = A \cap f^{-1}(L)$.
Then $K \in \Bb(A)$ and thus $K$ is a Souslin space. Now it suffices to apply \cite[Theorem~XIV.3.1]{k-m}
to obtain a function $h\dd L \to K$ such that $f \circ h$ is the identity map on $L$ and for every open in $K$
set $U \subset K$, $h^{-1}(U)$ is a member of the $\sigma$-algebra generated by the family of all Souslin subsets of $L$.
This implies that for every Borel subset $B$ of $K$, $h^{-1}(B)$ is absolutely measurable and therefore there is a Borel
function $w\dd L \to K$ and a set $B_0 \in \Bb(L)$ such that $\nu(B_0) = 0$ and $w = h$ on $L \setminus B_0$.
Now put $u = w \circ v$. By \LEM{AMimage}--(A), $u \in M^r_{\mu}(X)$. What is more, $f \circ u$ is $\mu$-almost everywhere
equal to $f \circ h \circ v = v$, which finishes the proof.
\end{proof}

Under the notation of \PRO{image} we get

\begin{cor}{image}
\begin{enumerate}[\upshape(i)]
\item If $X$ is a So-space, then $$\im M^r_{\mu}(f) = \bigcup_A M^r_{\mu}(f(A))$$ where $A$ runs over all separable
   closed subsets of $X$.
\item If for every compact subset $L$ of $\im f$ there is a (separable) Souslin subset $K$ of $X$ such that
   $L \subset f(K)$, then $\im M^r_{\mu}(f) = M^r_{\mu}(\im f)$.
\end{enumerate}
\end{cor}

The above result leads to the following

\begin{dfn}{s-map}
A map $f\dd X \to Y$ is said to be an \textit{s-map} if $f$ satisfies the assumption of the point (ii) of \COR{image}.
\end{dfn}

Basic examples of s-maps are closed maps whose domains are So-spaces and proper maps.\par
Now applying the General Scheme and main ideas of Section~3 of \cite{pn} (with the same functor $M$),
thanks to the homeomorphism extension theorem proved in \cite{chapman}, we easily obtain

\begin{thm}{functor}
Let $\Omega$ be a topological space homeomorphic to a nonseparable Hilbert space. Let $\ZzZ$ be the family (category)
of maps between $Z$-sets of $\Omega$ consisting of all pairs $(\varphi,L)$, where $\dom \varphi$, i.e. the domain
of $\varphi$, and $L$ are $Z$-sets of $\Omega$ and $\varphi$ is an $L$-valued continuous function. There is a functor
$\ZzZ \ni (\varphi,L) \mapsto \widehat{\varphi}_L \in \CCc(\Omega,\Omega)$ of extension which satisfies all the claims
of the points \textup{(a)}, \textup{(b)}, \textup{(h)}, \textup{(i)} stated on pages~1--2 of \textup{\cite{pn}}
and the claims of the points \textup{(d)}, \textup{(f)} and \textup{(g)} (of \textup{\cite{pn}}) concerning closures
of images. The functor preserves the properties of being an injection, an embedding or a map with dense image;
and satisfies all the claims of the points \textup{(c)--(g)} of \textup{\cite{pn}} for any s-map $\varphi$.
\end{thm}

\begin{rem}{prob}
Analogous functor as in \THM{functor} can be built using the functor $\widehat{P}$ studied
by Banakh\cite{banakh1,banakh2} and Banakh and Radul\cite{br}. (For a metrizable space $X$, $\widehat{P}(X)$ is the space
of all Borel probabilistic measures supported on $\sigma$-compact subsets of $X$ and for a map $f\dd X \to Y$
and $\mu \in \widehat{P}(X)$, $\widehat{P}(f)(\mu)$ is the transport of $\mu$ under $f$.) Theorem~2.11 of \cite{br}
says that $\widehat{P}(X)$ is homeomorphic to an infinite-dimensional Hilbert space, provided $X$ is completely
metrizable and noncompact. Thus it is enough to apply General Scheme of \cite{pn} and results
of Banakh\cite{banakh1,banakh2} on extending maps and bounded metrics via the functor $\widehat{P}$.
\end{rem}

We end the paper with the following two questions.\par
\textbf{Question 1.} Is $M(M(X))$ homeomorphic to $M(X)$ for an arbitrary metrizable space $X$~?\par
\textbf{Question 2.} Is $M_{\mu}(X)$ homeomorphic to $M^r_{\mu}(X)$ for an arbitrary metrizable space $X$
and any finite measure space $(\Omega,\Mm,\mu)$~?\par
Note that the affirmative answer for Question~2 implies the affirmative one for Question~1.

\end{document}